\documentclass[10pt,a4paper]{article}
\usepackage[utf8]{inputenc}
\usepackage{amsmath, amsthm}
\usepackage{amsfonts}
\usepackage{mathrsfs}
\usepackage{amssymb}
\usepackage{soul}
\usepackage{bbm}
\usepackage{hyperref}	
\usepackage{todonotes}
\usepackage{soul}

\newcommand{\snabla}{\slashed{\nabla}}

\newcommand{\sD}{\slashed{\Delta}}

\newcommand{\s}{\mathbb{S}}
\newcommand{\la}{\langle}
\newcommand{\ra}{\rangle}
\newcommand{\dPhiz}{\dot{\Phi}_0}

\newtheorem{theo}{Theorem}

\usepackage{slashed}

\usepackage[margin=1in]{geometry}

\usepackage{authblk}
\usepackage{enumerate}

\usepackage{enumerate}

\theoremstyle{plain}
\newtheorem{thm}{Theorem}[section]
\newtheorem*{thmunb}{Theorem}

\newtheorem{lemma}[thm]{Lemma}

\newtheorem{prop}[thm]{Proposition}

\newtheorem{conjecture}[thm]{Conjecture}

\theoremstyle{remark}
\newtheorem{rmk}{Remark}[section]

\theoremstyle{definition}

\newtheorem{defn}{Definition}[section]

\newcommand{\R}{\mathbb{R}}
\newcommand{\N}{\mathbb{N}}

\newcommand{\tpsi}{ \tilde{\psi}_{\alpha_0}}

\newcommand{\rd}{\partial}

\newcommand{\HHH}{ \underline{H}^{2}(\Sigma_T)}
\newcommand{\HH}{ \underline{H}^{1}(\Sigma_T)}

\newcommand{\HHd}{ \underline{H}^{-1}(\Sigma_T)}
\newcommand{\HHzd}{ \underline{H}^{-1}(\Sigma_T)}

\newcommand{\HHz}{ \underline{H}^{1}(\Sigma_T)}

\newcommand{\LL}{ \underline{L}^{2}(\Sigma_T)}
\newcommand{\LLi}{ \underline{L}^{2}(\Sigma_{-1})}

\theoremstyle{plain}

\theoremstyle{remark}

\theoremstyle{definition}

\usepackage{bm}

\usepackage{color}

\usepackage{float}

\addtocounter{tocdepth}{-2}
\usepackage{graphicx}
\usepackage{authblk}
 
\numberwithin{equation}{section}

\begin{document}
	
	\title{Late-time tails for scale-invariant wave equations with a potential and the near-horizon geometry of null infinity}
	\author[1]{Dejan Gajic}
	\author[2]{Maxime Van de Moortel}

	\affil[1]{\small Institut fur Theoretische Physik, Universit\"{a}t Leipzig, Br\"{u}derstrasse 16, 04103 Leipzig,
		Deutschland}
	\affil[2]{\small Department of Mathematics, Rutgers University, Hill Center, 110 Frelinghuysen Road, Piscataway, NJ, USA}
	
	\maketitle
	\abstract{We provide a definitive treatment, including sharp decay and the precise late-time asymptotic profile, for generic solutions of linear wave equations with a (singular) inverse-square potential in $(3+1)$-dimensional Minkowski spacetime. Such equations are scale-invariant and we show their solutions decay in time at a rate 
		determined by the coefficient in the inverse-square potential.

		We present a novel, geometric, physical-space approach for determining  late-time asymptotics, based around embedding Minkowski spacetime conformally into the  spacetime $AdS_2 \times \mathbb{S}^2$ (with $AdS_2$  the two-dimensional anti de-Sitter spacetime) to turn a global late-time asymptotics problem into a local existence problem for the wave equation in $AdS_2 \times \mathbb{S}^2$. Our approach is inspired by the treatment of the near-horizon geometry of extremal black holes 		in the physics literature.
		
		We moreover apply our method to another scale-invariant model: the (complex-valued) charged  wave equation on Minkowski spacetime in the presence of a static  electric field, which can be viewed as a simplification  of the charged Maxwell--Klein--Gordon equations on a black hole spacetime.}

	\setcounter{tocdepth}{2}

	\section{Introduction}\label{section.intro}
	
	The  wave equation  with an inverse-square potential in the $(3+1)$-dimensional Minkowski spacetime takes the  form  \begin{equation} \label{ISP} \tag{ISP}
		(m^{-1})^{\mu \nu} \partial_{\mu} \partial_{\nu} \phi= \frac{a}{r^2} \phi,	\end{equation} where  $a \in \mathbb{R}$, and $m =-dt^2 +dr^2 +r^2 d\sigma_{\mathbb{S}^2}$  is the Minkowski metric on $\R^{3+1}$ (expressed in spherical coordinates). The equation \eqref{ISP} appears in many problems in physics, and is also  a model problem for wave equations on manifolds with conical ends, see \cite{scalecrit1,scalecrit2,scalecrit2.5,scalecrit3}. It is \emph{scale-invariant}, in the sense that if $\phi$ is a solution to \eqref{ISP}, 
	$\phi_{\lambda}(t,r,\theta,\varphi)=\lambda^{-2} \phi(\lambda t, \lambda r, \theta,\varphi)$ is also a solution to \eqref{ISP}.

	Our main motivation to study scale-invariant equations originates from the propagation of charged scalar fields\footnote{These play an important role as a model of spherical  gravitational collapse, see for example \cite{r=0,JonathanICM,Kommemi} and the references therein. We will however not impose any symmetry assumption in the present work.} on a black hole. It turns out their governing equation -- the (charged)  Maxwell--Klein--Gordon system \cite{Moi2} -- is \emph{also scale-invariant} and resembles \eqref{ISP} (see Section~\ref{BH.intro} for an extended discussion of the various charged wave equations models). Here, we  consider a simplified model of the Maxwell--Klein--Gordon system:  the  complex-valued charged wave equation on Minkowski spacetime (with a non-trivial static electromagnetic field $F$),  which is the following  variant of \eqref{ISP} 
	\begin{equation} \label{CSF}\tag{CSF}
		m^{\mu \nu} D_{\mu} D_{\nu} \phi=0,
	\end{equation} where  the gauge derivative $D_{\mu}$ is defined as follows, introducing a coupling constant $q \in \R -\{0\} $: \begin{align}
		\label{def.gauge.deriv}	D_{\mu} = \partial_{\mu}+ iq A_{\mu}, \\
		\label{def.pot}	dA = F, \\
		\label{Maxwell.static}	F=  \frac{e}{r^2}\ dt \wedge dr.
	\end{align} 
	The above model is also scale-invariant and  describes a charged scalar field interacting with a fixed electromagnetic field $F=~ \frac{e}{r^2} dt \wedge~ dr$, where $e\in \R-\{0\}$. 
	Note that \eqref{def.pot}, \eqref{Maxwell.static} determine $A$ up to an exact 1-form $df$. Fixing $A$ amounts to making a so-called \emph{gauge choice}, and as it turns out, \eqref{CSF} is invariant under the following \emph{gauge transformation} on $(\phi,A)$:

	\begin{align*}	 A \rightarrow A+df,\ 
		\phi \rightarrow e^{-iq f} \phi .
	\end{align*}
	Observe that the modulus $|\phi|$ is invariant under this transformation. We will now formulate a rough version of our main result, in a specific gauge choice on which we will elaborate later on.  
	\begin{thmunb}[Rough version]
		\label{thm:rough}
		Assume that $|q_0e| \in (0,\frac{1}{2})$ for \eqref{CSF} and $a>-\frac{1}{4}$ for \eqref{ISP} and introduce the following exponents for any $\ell\in \N \cup\{0\}$:
		\begin{align*}
			p_{\ell}=&\:\frac{1}{2}\left(1+\sqrt{1-4(qe)^2+4\ell(\ell+1)}\right)+ iqe  \text{ for the equation \eqref{CSF}},\\
			p_{\ell}=&\: \frac{1}{2}\left(1+\sqrt{1+4 a+4\ell(\ell+1)}\right)  \text{ for the equation \eqref{ISP}}.
		\end{align*}
		
		Let $\phi$ be the unique solution of \eqref{ISP} or \eqref{CSF} with suitably regular and decaying initial data, such that the projection of $\phi$ onto spherical harmonics of degree $\geq \ell$ satisfies additionally the boundary conditions $\phi_{\geq \ell}\sim r^{-1+p_{\ell}}$ as $r\downarrow 0$. Then (in a suitable gauge choice $A$ for \eqref{CSF}), there exist uniform constants $C,\nu>0$, such that for all $(t,r,\theta,\varphi)\in \R_+\times \R_+\times \s^2$:
		\begin{enumerate}[\rm (i)]
			\item \begin{align} \label{Dtail}
				\left|r \phi_{\geq \ell}(t,r,\theta,\varphi)-Q_{\ell}(\theta,\varphi) \left(\frac{r}{(t-r)(t+r)}\right)^{p_{\ell}}\right|\leq C D_{\ell}[\phi]\left(\frac{r}{(t-r)(t+r)}\right)^{p_{\ell}}(t-r)^{-\nu},
			\end{align}
			with $D_{\ell}[\phi]$ an appropriate initial data energy norm and with $Q_{\ell}$ a function supported only on spherical harmonics of degree $\ell$.

			\item For generic initial data with respect to a suitably weighted higher-order energy norm, $Q_{\ell} \neq 0$, so \eqref{Dtail} gives the exact leading-order late-time asymptotics of $\phi$.
		\end{enumerate}		
	\end{thmunb}
	
	A more detailed version of the theorem is given in Section~\ref{section.precise} (Theorem \ref{main.theorem}). The reader who is not interested in the improved decay for higher spherical harmonics can directly consult the simplified version given in Theorem \ref{main.theoremell0}.
	
	\paragraph{Rate doubling and the conformal structure of null infinity.} The most striking phenomenon revealed by the above theorem (applied at $\ell=0$ in this discussion for simplicity)  is arguably the relation between the decay of the scalar field  $\phi$ on any time-like curve
	$\{R=R_0\}$ and the decay of the radiation field $(r\phi)_{|\mathcal{I}^+}(t-r,\omega):= \underset{(t,r)\rightarrow +\infty,\ t-r=cst}{\lim} (r\phi)(t,r,\omega)$ with $\omega=(\theta,\varphi)\in \mathbb{S}^2$, namely
	$$ \phi(t,R_0,\omega) \sim t^{-2p_{0}} ,\ (r\phi)_{|\mathcal{I}^+}(t-r,\omega) \sim (t-r)^{-p_{0}}\quad (t\to \infty).$$ 
	
	We also remark that this doubling property has already been established in a mathematically rigorous framework in the context of general asymptotically inverse-square potentials on black hole spacetimes in \cite{inverse.Dejan,Hintz.new} and for \eqref{ISP} in \cite{Dean}. We refer to Section~\ref{previous.intro} for an extended discussion on related works.  We are, however, not aware of any analogous previous results in the context of the charged wave equation \eqref{CSF}.

	\paragraph{The restriction in the parameters range for $a$ and $qe$.}
	
	In the ranges $a<-\frac{1}{4}$ and $|q e|>\frac{1}{2}$, $p_0 \in \mathbb{C}$ and our energy method break down. In this range, \eqref{ISP} and \eqref{CSF} are expected to be locally ill-posed, due to the singularity of the equation at the center $\{r=0\}$. This ill-posedness is, however, specific to singular equations such as \eqref{ISP}, \eqref{CSF} and does not apply to  wave equations with a regular potential. Such is the case of the Maxwell--Klein--Gordon system (motivating our study of \eqref{CSF}) or \eqref{ISP} on a black hole, in which the ranges  $|q e|>\frac{1}{2}$ and $a<-\frac{1}{4}$, however, result in a loss of positivity of natural energy quantities, which is related to the phenomenon of \emph{superradiance}.
	\paragraph{Gauge choice and resemblances between \eqref{ISP} and \eqref{CSF}.}
	The following gauge choice for \eqref{CSF} is very standard and is often chosen for its simplicity \begin{equation}\label{wrong.gauge}
		A = \frac{e}{r} (dt -dr).
	\end{equation}
	
	With this gauge choice, \eqref{CSF} takes an analogous  form   to \eqref{ISP} (which is obviously also scale-invariant), with  a complex constant $iq e$ multiplying the $r^{-2}$ potential and extra first order terms:
	\begin{equation}\label{CSF.wrong.gauge}
		(m^{-1})^{\mu \nu} \partial_{\mu} \partial_{\nu} \phi=  \frac{iqe}{r^2} \phi+ \frac{2iqe}{r} (\partial_t \phi+\partial_r \phi).
	\end{equation}
	
	Since \eqref{CSF} is gauge-invariant, and resembles \eqref{ISP} in the above gauge choice, it makes sense that both equations can be treated with similar methods. However, as will become clear later (see Section \ref{section.key}), the above gauge choice, while simple, is particularly inconvenient in our approach. Instead, we will fix the following convenient gauge choice for the rest of the paper, differing from the above by the closed form  $- 2e\ d ( \ln(t+r))$ (see   \eqref{gauge.Av}, \eqref{gauge.Au} which is the same gauge choice as \eqref{gauge}): \begin{equation}\label{gauge} \tag{G}
		A = \frac{e}{r} (dt -dr)- 2e\ d ( \ln(t+r)).
	\end{equation}
	In particular, \eqref{gauge} is the gauge choice in which the above theorem is valid (see also in Theorem~\ref{main.theorem}).

	\paragraph{Strategy to obtain decay-in-time.} 
	Our approach relies on three key ideas, which we will further develop in Section~\ref{section.key}: \begin{enumerate}[A.]
		\item \label{step1} We conformally embed the Minkowski spacetime excluding the line $\Gamma=\{x^1=x^2=x^3=0\}$ $(m,\R^{3+1}- \Gamma)$ into a local patch of the Bertotti--Robinson spacetime $(r^{-2}m, AdS_2 \times \mathbb{S}^2)$. We then express \eqref{ISP} and \eqref{CSF} in new coordinates $(T,R,\theta,\varphi)$ that cover the spacetime with \emph{finite range}. These new coordinates arise from the observation that we can \emph{translate} the embedded Minkowski spacetime, viewed as a subset of $AdS_2 \times \mathbb{S}^2$(see Figure~\ref{fig:penrose1}) in the global $AdS_2\times \s^2$ time direction to obtain \emph{another}, overlapping embedded copy of Minkowski spacetime, as depicted in Figure~\ref{fig:penrose2}. As it turns out, this procedure reduces the global problem of decay for \eqref{ISP} and \eqref{CSF} to a \emph{local boundedness problem} for a Klein--Gordon equation on  $AdS_2 \times \mathbb{S}^2$.

		\item  \label{step2} We solve the Klein--Gordon equation locally on  $AdS_2 \times \mathbb{S}^2$ by renormalizing the scalar field $\phi$ to a new quantity $\Phi$, which is defined as follows (with the understanding that $q=0$ for \eqref{ISP}): $$ \Phi(R,T,\theta,\varphi) = R^{-\frac{1}{2} -i qe} r\phi $$ solving an equation involving twisted-derivative operators $ \partial_R^{\alpha}f=   R^{-\alpha}\partial_R ( R^{\alpha} f)$  for some $\alpha>0 $: \begin{equation}\label{wave.ads.intro}
			-\partial_T^2 \Phi + \partial_R^{1-\alpha}\partial_R^{\alpha} \Phi + R^{-2} \sD_{\mathbb{S}^2}\Phi=0.
		\end{equation} 
		
		\item \begin{enumerate}[i.] \label{step3}
			\item  We reduce the proof of determining the precise late-time behaviour of $\phi$ to showing that $R^{-\alpha}\Phi(i^+) \neq~0$ for generic initial data, where $i^+$ is timelike infinity, interpreted as a  sphere $\{R=T=0\}$ on the $AdS_2\times \s^2$ boundary. 
			\item \label{step4} To obtain asymptotics for the angular modes $\phi_{\geq \ell}$, we require more regularity on the initial data and a stronger fall-off in $r$ at the center $\{r=0\}$. We quantify this regularity and fall-off with appropriate twisted-Sobolev spaces.
		\end{enumerate}

	\end{enumerate} 
	\begin{figure}[H]
		\begin{center}
			\includegraphics[scale=0.6]{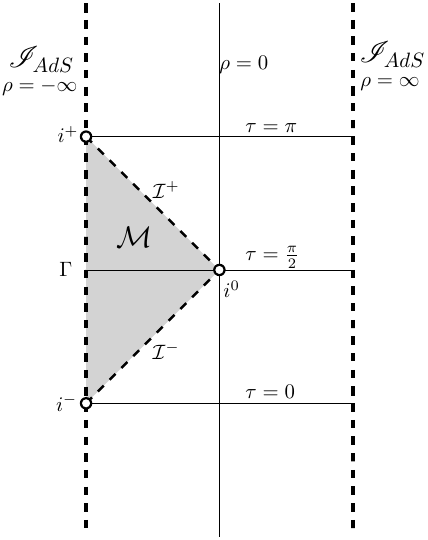}
		\end{center}
		\caption{A 2-dimensional representation of the Minkowski region $(\mathcal{M}=\R^{3+1}- \Gamma,m)$ embedded in $AdS_2\times \s^2$.} 
		\label{fig:penrose1}
	\end{figure}
	
	Step \ref{step1} is inspired by a construction that is frequently used in the physics literature to capture the dynamics of extremal black holes that are inherent to the black hole event horizon, called \emph{the near-horizon geometry}, see e.g.\  \cite{HadarReall,Zimmerman1,Zimmerman2,Zimmerman3}. Near-horizon geometries are obtained via a limiting procedure by blowing up a neighbourhood of the event horizon in appropriate rescaled coordinates. In the case of extremal Reissner--Nordst\"om black holes, this procedure results in the Bertotti--Robinson spacetimes $AdS_2 \times \mathbb{S}^2$ (see Section~\ref{section.key}). In view of the existence of a conformal isometry that maps neighbourhoods of extremal Reissner--Nordst\"om event horizons to neighbourhoods of null infinity in conformal extensions of the spacetime, see \cite{CT84}, it is perhaps not surprising that the near-horizon geometry $AdS_2\times \s^2$ is also relevant to capture dynamics inherent to null infinity. The main motivation of the present article is to understand in a mathematically rigorous setting the application of near-horizon geometries to capture key dynamical aspects in the simpler setting of \eqref{ISP}/\eqref{CSF} in which a black hole is absent.	
	
	Step \ref{step2} is inspired by the techniques of Warnick \cite{WarnickAdS}, where $\alpha$-twisted energy estimates were introduced to address the local well-posedness of the Klein--Gordon equation \eqref{wave.ads.intro} on asymptotically anti de Sitter spacetimes with different boundary conditions at the conformal anti de Sitter boundary. Since \eqref{CSF} features additional first order terms, the work in the present papers additionally involves the introduction of a Hardy-type inequality that is adapted to the setting $\alpha$-twisted derivatives.

	\paragraph{Outline of the introduction} In the subsequent Section~\ref{section.precise}, we provide a more detailed version (namely: Theorem~\ref{main.theoremell0} \& Theorem~\ref{main.theorem}) of our aforementioned main result and its assumptions. In Section~\ref{BH.intro}, we discuss the similarities between \eqref{CSF} and the Maxwell--Klein--Gordon system on a black hole, for which the asymptotics \eqref{Dtail} are also conjectured to hold. In Section~\ref{previous.intro}, we discuss previous works on \eqref{ISP} establishing versions of \eqref{Dtail}. In Section~\ref{section.key}, we discuss our new method based on the  $AdS_2\times \s^2$ geometry to establish \eqref{Dtail} and its connection to the near-horizon geometry in extremal black holes. In Section~\ref{NP.intro}, we  finally contrast \eqref{Dtail} with  other decay rates obtained for wave equations  with $O(r^{-p})$ potentials, including the (uncharged) wave equation on a   black hole.
	
	\subsection{Precise statement of the main results}\label{section.precise}
	
	\paragraph{Set-up of the initial data}
	
	We will set initial Cauchy data on the spacelike hyperboloid $$\Sigma_{-1}= \{ (t-2)^2 -r^2 =4,\ t\geq 4\}$$ which penetrates into null infinity $\mathcal{I^+}$, see Figure~\ref{fig:penrose2}. The motivation for considering $\Sigma_{-1}$ will appear more clearly in Section~\ref{section.prelim}, where it is showed that $\Sigma_{-1} = \{T=-1\}$ in the $(T,R,\theta,\varphi)$ coordinates on $AdS_2 \times \mathbb{S}^2$. The analysis will be restricted to $D^+(\Sigma_{-1})=\{ (t-2)^2 -r^2  \geq 4,\ t\geq 4\}$, which is the {\it domain of dependence} of $\Sigma_{-1}$.
	
	Denote the conformally-Killing Morawetz vector field  $K$ (which is timelike and normal to $\Sigma_{-1}$) and its Lorentzian orthogonal $K^{\perp}$ (which is spacelike, spherically symmetric and tangent to $\Sigma_{-1}$). These vector fields are defined as follows in terms of standard spherical coordinates $(t,r,\theta,\varphi)$: \begin{align} \label{K.def}
		K&:= \frac{t^2+r^2}{2} \partial_t + r t\ \partial_r,\\ K^{\perp} &:=   r t\ \partial_t +  \frac{t^2+r^2}{2} \partial_r. \label{Kperp.def}
	\end{align}

	We shall  prescribe initial data $(\phi^0,\dot \phi^0)$ for \eqref{CSF} (or \eqref{ISP}) with $\dot \phi^0= K \phi^0$ in evolution. 
	
	\paragraph{Weighted Sobolev spaces.} We will now define the main weighted Sobolev spaces, which contain the relevant initial data. Consider the space $L^2(\Sigma_{-1},d\mu)$, with volume form $d\mu= \frac{r}{(1+r)^3} drd\sigma_{\s^2}$. \begin{defn}\label{def:weightsobspace}
		We  define the  weighted Sobolev spaces: $\mathcal{H}^0(\Sigma_{-1})=L^2(\Sigma_{-1},d\mu)$ and for $1\leq k\leq2$
		\begin{align}
			\label{H1.def}	\mathcal{H}^k(\Sigma_{-1}):=&\:\Biggl\{f\in L^2(\Sigma_{-1},d\mu)\,\Big |\, (1+r^{-1})f\in 	\mathcal{H}^{k-1}(\Sigma_{-1}),\\ \nonumber
			&((1+r^{-1})\slashed{\nabla}_{\s^2})^{k_1}(K^{\perp})^{k_2}f\in L^2(\Sigma_{-1},d\mu)\quad \forall\: k_1+k_2\leq k \Biggr\},\\
			||f||_{\mathcal{H}^k(\Sigma_{-1})}^2:=&\: \|(1+r^{-1})f\|_{\mathcal{H}^{k-1}(\Sigma_{-1})}^2+\sum_{0\leq k_1+k_2\leq k}\|((1+r^{-1})\slashed{\nabla}_{\s^2})^{k_1}(K^{\perp})^{k_2}f\|_{\mathcal{H}^{0}(\Sigma_{-1})}^2.\nonumber
		\end{align}

		We moreover define $\mathcal{H}^2_{0,0}(\Sigma_{-1})\subset \mathcal{H}^2(\Sigma_{-1})$ as follows:
		\begin{align}
			\label{H1.0.0.def}	\mathcal{H}^2_{0,0}(\Sigma_{-1}):=&\:\left\{f\in \mathcal{H}^1(\Sigma_{-1}),\   \pi_{\geq 1}(f)\in    \mathcal{H}^2(\Sigma_{-1}),\  d_{R}^{1-\alpha_{\ell}}d_{R}^{\alpha_{\ell}}\left[f\right] \in 	\mathcal{H}^0(\Sigma_{-1})  \right\},\\
			||f||_{\mathcal{H}^2_{0,0}(\Sigma_{-1})}^2:=&\: \|\pi_{\geq 1}f\|_{\mathcal{H}^2(\Sigma_{-1})}^2+\| d_R^{1-\alpha_{\ell}}d_R^{\alpha_{\ell}}\left[f\right] \|^2_{ \mathcal{H}^0(\Sigma_{-1})},\nonumber
		\end{align}
		where we define the  differential operator (which agrees with $\rd_R^{\alpha}$ defined in Section~\ref{section.prelim}) for any $\alpha \in \R$: \begin{equation}\label{R.formula}\begin{split} &
				d_R^{\alpha}[F] =  [R(r)]^{-\alpha} K^{\perp} \left( [R(r)]^{\alpha} F\right) ,\\ & R(r):= \frac{r}{2+\sqrt{4+r^2}}.
			\end{split}
		\end{equation} 
		(see Section~\ref{section.prelim}, where $R$ is defined to coincide with \eqref{R.formula} on $\Sigma_{-1}$).

	\end{defn}
	
	\begin{rmk} We  explain and motivate the choice of the above vector fields and Sobolev spaces.  \begin{enumerate}[A.]
			
			\item The Morawetz vectors fields $K$ and $K^{\perp}$ defined  in \eqref{K.def} and \eqref{Kperp.def} will become very natural when working later with $(T,R)$ coordinates in that $K=\partial_T$ and $K^{\perp}=\partial_R$ (see Section~\ref{section.prelim}). In particular, note that $K^{\perp}$ is tangential to $\Sigma_{-1}$, while $K$ is orthogonal to $\Sigma_{-1}$.
			\item The volume form $d\mu= \frac{r}{(1+r)^3}drd\sigma_{\s^2}$ may look unusual at first glance but is, in fact, proportional to $RdR  d\sigma_{\s^2}$, which is a natural volume form on $\Sigma_{-1}=\{T=-1\}$, see Section~\ref{section.prelim}.
			\item  From our analysis, it can be inferred that for  $\phi$ a solution of \eqref{CSF}/\eqref{ISP} with sufficiently regular initial data,  $(K^{\perp})^n\psi$ is bounded towards future null infinity $\mathcal{I}^+$ for any $n\in \mathbb{N}$ (in fact this property also holds in  the more involved context  of Maxwell--Klein--Gordon on a black hole  \cite{Moi2}).  $\psi$, $\snabla_{\mathbb{S}^2} \psi$ and $K^{\perp} \psi$  have a non-zero limit  at $\mathcal{I}^+$, and thus it makes sense that these three quantities  are involved (with no other scaling factor as $r\rightarrow+\infty$) in Definition \ref{def:weightsobspace}. 
			\item  
			Towards $\{r=0\}$ however, the $r^{-1}$ weights are important in that higher Sobolev regularity is associated to higher orders of vanishing at $\{r=0\}$; this property turns out to be crucial to prove decay in time of the higher angular modes, for which the data needs to be more regular. 
			\item Finally, the space $\mathcal{H}^2_{0,0}(\Sigma_{-1})$ is designed to coincide with $\underline{H}^2_{0,0}(\Sigma_{T})$  defined in Section~\ref{sec:functionspaces} for $T=-1$.
		\end{enumerate}
		
	\end{rmk}
	As we can be inferred from the above definitions, we  measure regularity in terms of the weighted quantity  $\psi:=r\phi$ and not of the scalar field $\phi$ itself. The limit $r\phi|_{\mathcal{I^+}}$ is the Friedlander radiation field. This observation motivates the  initial data regularity classes in Definition \ref{def:regclasses} below, which are designed so that they are consistent with natural initial data spaces in $(T,R)$ coordinates introduced in Section~\ref{sec:functionspaces} below. Following this logic (to match with  \eqref{T.deriv}), we define recursively, for any couple of functions $(\Phi^0, \dot{\Phi}^0)$ on $\Sigma_{-1}$: $K\Phi^0 := \dot{\Phi}^0$ and for $n \geq 2$: \begin{equation}
		K^{n}\Phi^0=\left[d_R^{1-\alpha_0}d_R^{\alpha_0}+[R(r)]^{-2}\slashed{\Delta}_{\s^2}\right]K^{n-2}\Phi^0-2i qe [R(r)]^{-1}K^{n-1}\Phi^0.
	\end{equation}
	
	\begin{defn}
		\label{def:regclasses}
		We define the following regularity class for initial data $(\phi^0,\dot{\phi}^0)$ on $\Sigma_{-1}$: denote $\Phi^0=[R(r)]^{-\frac{1}{2}-ie}\phi^0(r)$ and $\dot{\Phi}^0=[R(r)]^{-\frac{1}{2}-ie}\dot{\phi}^0(r)$, then \begin{align*}
			& {\mathcal{H}}^1_{{\rm data}} := \frac{[R(r)]^{\frac{1}{2}+i q e}}{r}\  {\mathcal{H}}^1(\Sigma_{-1})\times \frac{[R(r)]^{\frac{1}{2}+iqe}}{r}\  {\mathcal{H}}^0(\Sigma_{-1}),\\   &	\mathcal{H}^2_{{\rm data},N}= \bigcap_{n=0}^N \{ (\Phi^0,\dot \Phi^0), \, K^{n}\Phi^0 \in \mathcal{H}^2_{0,0}(\Sigma_{-1}),\, K^{n+1}\Phi^0\in {\mathcal{H}}^1(\Sigma_{-1}) \} \text{ for any } N \in \mathbb{N}\cup\{0\},\\  & {\mathcal{H}}^2_{{\rm data},N, \snabla} := \{ (\Phi^0,\dot \Phi^0)\in \mathcal{H}^2_{{\rm data},N},\  \left(\snabla_{\mathbb{S}^2}^j\Phi^0 , \snabla_{\mathbb{S}^2}^j\dot{\Phi}^0\right)\in  \mathcal{H}^2_{{\rm data},N}\text{ for all } 0 \leq j \leq 2   \} \text{ for any }N \in \mathbb{N}\cup\{0\},
		\end{align*} with associated norms 
		
		\begin{align*}
			&  \| (\phi^0,\dot{\phi}^0)\|_{{\mathcal{H}}^1_{\rm data}}^2 =  \| \Phi^0 \|_{ {\mathcal{H}}^1(\Sigma_{-1})}^2+  \| \dot{\Phi}^0 \|_{ {\mathcal{H}}^{0}(\Sigma_{-1})}^2,\\ & \| (\phi^0,\dot{\phi}^0)\|_{{\mathcal{H}}^2_{\rm data,N}}^2 =  \sum_{n=0}^N\| K^{n} \Phi^0 \|_ {\mathcal{H}^2(\Sigma_{-1})}^2+  \| K^{n+1} \dot{\Phi}^0 \|_{ {\mathcal{H}}^{1}(\Sigma_{-1})}^2,\\
			&  \| (\phi^0,\dot{\phi}^0)\|_{{\mathcal{H}}^k_{{\rm data},N,\snabla}}^2 = \sum_{0\leq j\leq 2} \| (\snabla_{\mathbb{S}^2}^j\phi^0,\snabla_{\mathbb{S}^2}^j\dot{\phi}^0)\|_{{\mathcal{H}}^2_{\rm data,N}}^2 .
		\end{align*}
	\end{defn}
	
	We introduced ${\mathcal{H}}^2_{{\rm data},N,\snabla}$ to quantify higher angular regularity; this is simply to be able to obtain pointwise bounds, as opposed to square-averages over $\mathbb{S}^2$ if we only assume data is in $\mathcal{H}^2_{\rm data,N}$.

	We will now introduce a notion of weak solution with initial data on $\Sigma_{-1}$, which will become natural when working with the coordinates $(T,R,\theta,\varphi)$ introduced in Section~\ref{section.prelim}.\begin{defn} We say $\phi$ is a weak solution of \eqref{ISP}/\eqref{CSF} with initial data $(\phi^0,\dot \phi^0)\in \mathcal{H}^1_{data}$ if its renormalization $\Phi$ (defined in \eqref{Phi.def}) satisfies the requirements of Definition~\ref{def;weaksol}.
	\end{defn}
	
	\paragraph{The main theorems.}
	We state here more precise versions of Theorem \ref{thm:rough}, which include precise expressions for the initial data norms appearing in the estimate and precise definitions of genericity.

	We first state a theorem \emph{without} any restrictions to higher spherical harmonic functions.
	
	\begin{theo}\label{main.theoremell0}
		Assume $|q_0e| \in (0,\frac{1}{2})$ for \eqref{CSF} and $a>-\frac{1}{4}$ for \eqref{ISP} and denote for $\ell\in \N_0$:
		\begin{align*}
			p_{0}=&\:\frac{1}{2}\left(1+\sqrt{1-4(qe)^2}\right)+ iqe  \text{ for the equation \eqref{CSF}},\\
			p_{0}=&\: \frac{1}{2}\left(1+\sqrt{1+4 a}\right)  \text{ for the equation \eqref{ISP}}.
		\end{align*}

		Define $\alpha_{0}:= \Re(p_{0})-\frac{1}{2}>0$. Assume $(\phi^0,\dot\phi^0) \in\mathcal{H}^{2 }_{\rm data,\lfloor  \alpha_{0}\rfloor+1} $. There exists a unique weak solution $\phi$  of \eqref{CSF} in the choice \eqref{gauge} (respectively of \eqref{ISP}) such that $(\phi, K\phi)|_{\Sigma_{-1}} = (\phi^0,\dot \phi^0)$, and there exist constants $C=C(\alpha_{0})>0$   such that for any  $0<\nu<\lfloor \alpha_{0}\rfloor-(\alpha_{0}-1)$:
		\begin{enumerate}[\rm (i)]
			\item  Along the spheres $S^2_{t_0,r_0}=\{t=t_0\}\cap\{r=r_0\}$ in $D^+(\Sigma_{-1})$, we have the $L^2$ estimate:
			\begin{equation}
				\label{main.theorem.sharp.l=0}
				\left\|  r\phi -Q_{0} \left(\frac{r}{(t-r)(t+r)}\right)^{p_{0}}\right\|_{L^2(S^2_{t,r})} \leq C \left(\frac{r}{(t-r)(t+r)}\right) ^{p_{0}}(t-r)^{-\nu}\|(\phi^0,\dot\phi^0)\|_{\mathcal{H}^{2 }_{\rm data,\lfloor  \alpha_{0}\rfloor+1}},
			\end{equation}
			where $Q_0\in \mathbb{C}$ satisfies:
			\begin{equation}
				\label{eq:boundQlell0}
				|Q_{0}|\leq C\|(\phi^0,\dot\phi^0)\|_{\mathcal{H}^{2 }_{\rm data,\lfloor  \alpha_{0}\rfloor+1}} .
			\end{equation}
			\item If $(\phi^0,\dot\phi^0) \in \mathcal{H}^{2 }_{\rm data,\lfloor  \alpha_{0}\rfloor+1,\snabla} $, then $\phi$ is a strong solution of \eqref{ISP}/\eqref{CSF} and we can also  estimate it pointwise: for all $(t,r,\theta,\varphi)\in D^+(\Sigma_{-1})$
			\begin{equation}
				\label{main.theorem.sharpLinftyell0}
				\left| r\phi(t,r,\theta,\varphi) -Q_{0}\left(\frac{r}{(t-r)(t+r)}\right)^{p_{0}}\right| \leq C \left(\frac{r}{(t-r)(t+r)}\right) ^{p_{0}}(t-r)^{-\nu}\|(\phi^0,\dot\phi^0)\|_{ \mathcal{H}^{2 }_{\rm data,\lfloor  \alpha_{0}\rfloor+1,\snabla}}  .
			\end{equation}
			\item There exists a linear subspace $\mathcal{E}_{0}\subset  \mathcal{H}^{2 }_{\rm data,\lfloor  \alpha_{0}\rfloor+1} $ of codimension greater or equal to $1$, such that under the assumption $(\phi^0,\dot\phi^0) \in\mathcal{H}^{2 }_{\rm data,\lfloor  \alpha_{0}\rfloor+1}\setminus \mathcal{E}_{0}$:
			\begin{equation*}
				Q_{0}\neq 0.
			\end{equation*}
		\end{enumerate}
	\end{theo}
	
	We can obtain a refinement of the above theorem, when we project $\phi$ onto spherical harmonics of degree $\geq\ell$: namely, we consider angular modes with angular momentum above a given threshold.
	\begin{theo}\label{main.theorem}
		Assume $|q_0e| \in (0,\frac{1}{2})$ for \eqref{CSF} and $a>-\frac{1}{4}$ for \eqref{ISP} and denote for $\ell\in \N_0$:
		\begin{align*}
			p_{\ell}=&\:\frac{1}{2}\left(1+\sqrt{1-4(qe)^2+4\ell(\ell+1)}\right)+ iqe  \text{ for the equation \eqref{CSF}},\\
			p_{\ell}=&\: \frac{1}{2}\left(1+\sqrt{1+4 a+4\ell(\ell+1)}\right)  \text{ for the equation \eqref{ISP}}.
		\end{align*}

		Define $\alpha_{\ell}:= \Re(p_{\ell})-\frac{1}{2}>0$. Assume $(\pi_{\geq \ell}(\phi^0),\pi_{\geq \ell}(\dot\phi^0)) \in \mathcal{H}^{2 }_{\rm data,\lfloor  \alpha_{\ell}\rfloor+1} $,
		where $\pi_{\geq \ell}$ is the projection on all spherical harmonics on $\mathbb{S}^2$ of degree $\geq \ell$ (see Section \ref{spherical.section}). Then there exists a unique weak solution $\phi_{\geq \ell}$ of \eqref{CSF} in the choice \eqref{gauge} (respectively of \eqref{ISP}) such that $(\phi_{\geq \ell}, K\phi_{\geq \ell})|_{\Sigma_{-1}} = (\phi^0,\dot \phi^0)$, and there exist constants $C=C(\ell,\alpha_{\ell})>0$  such that for any  $0<\nu<\lfloor \alpha_{\ell}\rfloor-(\alpha_{\ell}-1)$:
		\begin{enumerate}[\rm (i)]
			\item  Along the spheres $S^2_{t_0,r_0}=\{t=t_0\}\cap\{r=r_0\}$ in $D^+(\Sigma_{-1})$, we have the $L^2$ estimate:
			\begin{multline}
				\label{main.theorem.sharp}
				\left\|  r\phi_{\geq \ell}-Q_{\ell} \left(\frac{r}{(t-r)(t+r)}\right)^{p_{\ell}}\right\|_{L^2(S^2_{t,r})} \\
				\leq C \left(\frac{r}{(t-r)(t+r)}\right) ^{p_{\ell}}(t-r)^{-\nu}\|(\pi_{\geq \ell}(\phi^0),\pi_{\geq \ell}(\dot\phi^0))\|_{\mathcal{H}^{2 }_{\rm data,\lfloor  \alpha_{\ell}\rfloor+1} },
			\end{multline}
			where $Q_{\ell}\in L^{\infty}(\s^2)$ (interpreted above as a function on $S^2_{t,r}$ independent of $t,r$) is supported on spherical harmonics of degree $\ell$ and satisfies:
			\begin{equation}
				\label{eq:boundQl}
				\|Q_{\ell}\|_{L^{\infty}(\s^2)}\leq C\|(\pi_{\geq \ell}(\phi^0),\pi_{\geq \ell}(\dot\phi^0))\|_{\mathcal{H}^{2 }_{\rm data,\lfloor  \alpha_{\ell}\rfloor+1} }.
			\end{equation}
			\item If $(\pi_{\geq \ell}(\phi^0),\pi_{\geq \ell}(\dot\phi^0)) \in \mathcal{H}^{2 }_{\rm data,\lfloor  \alpha_{\ell}\rfloor+1,\snabla} $, then $\phi$ is a strong solution of \eqref{ISP}/\eqref{CSF} and we can also estimate it pointwise: for all $(t,r,\theta,\varphi)\in D^+(\Sigma_{-1})$
			\begin{multline}
				\label{main.theorem.sharpLinfty}
				\left| r\phi_{\geq \ell}(t,r,\theta,\varphi) -Q_{\ell}(\theta,\varphi) \left(\frac{r}{(t-r)(t+r)}\right)^{p_{\ell}}\right| \\
				\leq C \left(\frac{r}{(t-r)(t+r)}\right) ^{p_{\ell}}(t-r)^{\nu}\|(\pi_{\geq \ell}(\phi^0),\pi_{\geq \ell}(\dot\phi^0))\|_{\mathcal{H}^{2 }_{\rm data,\lfloor  \alpha_{\ell}\rfloor+1,\snabla}}.
			\end{multline}
			\item There exists a linear subspace $\mathcal{E}_{\ell}\subset\mathcal{H}^{2 }_{\rm data,\lfloor  \alpha_{\ell}\rfloor+1}$ of codimension greater or equal to $2\ell+1$, such that under the assumption $(\pi_{\geq \ell}(\phi^0),\pi_{\geq \ell}(\dot\phi^0)) \in \mathcal{H}^{2 }_{\rm data,\lfloor  \alpha_{\ell}\rfloor+1}\setminus \mathcal{E}_{\ell}$:
			\begin{equation*}
				Q_{\ell}\not \equiv 0.
			\end{equation*}
		\end{enumerate}
	\end{theo}

	\subsection{Charged scalar field equations on a black hole }\label{BH.intro}
	
	In this section, we discuss the following (massless)  Maxwell--Klein--Gordon model and its variants \begin{equation}\begin{split}\label{KG}
			&(g^{-1})^{\mu \nu} D_{\mu} D_{\nu} \phi=\:0,\\  & \nabla^{\nu} F_{\mu \nu}= q\ \Im(\bar{\phi} D_{\mu}\phi),\\ & D_{\mu} = \nabla_{\mu } + i q A_{\mu},\ F= dA. \end{split}
	\end{equation} where $\nabla_{\mu}$ is the covariant derivative of $g$   and $g=g_{RN}$  is the Reissner--Nordström exterior metric  \begin{equation}\begin{split}
			&g_{RN}=-\Big(1-\frac{2M}{r}+\frac{e^2}{r^2}\Big) dt^2 + \Big(1-\frac{2M}{r}+\frac{e^2}{r^2}\Big)^{-1} dr^2 + r^2\left( d\theta^2+ \sin^2(\theta) d\varphi^2\right),\\  & (t,r,\theta,\varphi) \in \R\times [r_+(M,e),+\infty)\times \mathbb{S}^2, \text{ where } r_+(M,e)=M+\sqrt{M^2-e^2}>0,\end{split}
	\end{equation} which represents a stationary charged black hole, with the sub-extremality condition $0 \leq |e|<M$.

	The Maxwell field $F$ is dynamical in \eqref{KG} and one expects\footnote{\cite{Moi2} in particular includes a proof of this statement for weakly charged spherically symmetric solutions of \eqref{KG}.} that for solutions  of \eqref{KG} with sufficiently  regular initial data, $F$ converges at late-time to the static \eqref{Maxwell.static} with $e\neq 0$ generically. 
	
	One can also study a simplified linear charged wave model where $\phi$ still satisfies \eqref{KG} but now $F$ is static given by \eqref{Maxwell.static}. In an appropriate gauge choice, this model takes the following form: 	 \begin{equation}\label{KG.gauged}
		(g_{RN}^{-1})^{\mu \nu} \partial_{\mu} \partial_{\nu} \phi= \left( \frac{iqe}{r^2}-\frac{2M}{r^3}+ \frac{2e^2}{r^4}\right) \phi+ \frac{2iqe}{r} (\partial_t \phi+\partial_r \phi),
	\end{equation} which is  identical to \eqref{CSF.wrong.gauge} (i.e.\ \eqref{CSF} in the gauge \eqref{wrong.gauge}), up to lower order perturbations $O(\frac{2M}{r^3})$ as $r\rightarrow+\infty$. As it turns out, \eqref{CSF}, \eqref{KG}, \eqref{KG.gauged} conjecturally share the \emph{same late-time asymptotics}:

	\begin{conjecture}\label{KG.conj} Let $0<|qe|<\frac{1}{2}$. A  solution $\phi$  of \eqref{KG.gauged} with generic, regular initial data obeys the asymptotics \eqref{main.theorem.sharp.l=0} of Theorem~\ref{main.theoremell0} and  \eqref{main.theorem.sharp} of Theorem~\ref{main.theorem} for $r\geq  r_+(M,e)$.
	\end{conjecture}
	
	Conjecture~\ref{KG.conj} is supported by heuristics/numerics works \cite{HodPiran2,HodPiran1} and was partially proven  with near-optimal upper bounds \cite{Moi2} obtained for weakly charged spherically symmetric solutions of \eqref{KG}.
	\begin{rmk}
		It  also makes sense to formulate an analogous  conjecture   to  Conjecture~\ref{KG.conj} for  the Einstein--Maxwell--Klein--Gordon system,  a model which is similar to \eqref{KG} except that the black hole metric $g$ is also dynamical, see  \cite{Moi,Moi4,r=0} for works  in the context of the black hole interior.
	\end{rmk}
	
	\begin{rmk}
		We also mention a vast body of works \cite{YangYu,ShiwuKG,LindbladKG,LinbladKG2,ShuKG} on the late-time asymptotics  on the Maxwell--Klein--Gordon system \eqref{KG} on Minkowski spacetime $g=m$. However, in contrast to  \eqref{KG} in the black hole setting, the Maxwell field $F$   disperses to $0$ at late times and therefore, the asymptotics are different from the ones of Conjecture~\ref{KG.conj} and Theorem\ref{main.theoremell0}/Theorem~\ref{main.theorem}.

	\end{rmk}

	The methods developed in the present paper do not apply to \eqref{KG} or \eqref{KG.gauged}. See however upcoming work \cite{inprep} where late-time tails on \eqref{KG.gauged} are obtained in the black hole setting, using the the framework initiated in \cite{inverse.Dejan}.

	\subsection{Wave equations with (scale-invariant) potentials} \label{previous.intro} 
	Time-decay results for \eqref{ISP} (or its variant with asymptotically inverse-square potential) have a rich history, see e.g.\ \cite{scalecrit1,scalecrit2,scalecrit2.5,scalecrit3}.  For a large class of wave equations with a \emph{smooth} potential $V(r)= a\ r^{-2}$ at large $r$, sharp point-wise localized decay estimates of the following form, for some large $\sigma_1, \sigma_2>0$: \begin{equation*}
		\| r^{-\sigma_1} \phi\|_{L^{\infty}(\Sigma_t)} \lesssim \| r^{\sigma_2} (\phi^0,\dot{\phi}^0)\|_{L^1(\Sigma_0)}\ [1+t]^{-p_0} \text{ as } t\rightarrow+\infty
	\end{equation*} were obtained using semi-classical analysis have been obtained by Costin--Schlag--Staubach--Tanveer \cite{Schlag_inverse}, with the same rate $p_0$ as in Theorem~\ref{main.theoremell0}, under the assumption, however, that $\sqrt{1+4a} \neq 2k+1$, for any $k\in \N$. An important difference between \cite{Schlag_inverse} and Theorem~\ref{main.theoremell0} is that we consider initial data with limited $r$-decay as $r\rightarrow+\infty$ and, in this situation, we show that $O(t^{-p_0})$ decay is sharp for \eqref{ISP}, \emph{even in the case $\sqrt{1+4a} =2k+1$ for some $k\in \N$}. If $\sqrt{1+4a+\ell(\ell+1)} =2k+1$, then $a+\ell(\ell+1)=(k+\ell)(k+\ell+1)$, so \eqref{ISP} restricted to the $\ell$th spherical harmonics is equivalent to the standard wave equation restricted to the $(\ell+k)$th spherical harmonics and the strong Huygens principle applies. This means that arbitrarily fast decay in $t$ for $\phi_{\ell}$ follows from suitably fast decay in $r$ and, in particular, compactly supported initial data leads to compact support in time; see also related estimates establishing $O(t^{-N})$ decay for any $N>0$  \cite{Schlag_inverse,wave.Schlag}.

	In \cite{Hintz.new}, Hintz derived precise late-time asymptotics for wave equations with asymptotically inverse-square potentials on a general class of stationary spacetimes.

	The aforementioned  recent work of Baskin--Gell-Redman--Marzuola \cite{Dean} obtained sharp decay  estimates for \eqref{ISP} with a result analogous to our Theorem~\ref{main.theoremell0}, albeit using a different conformal method coupled to microlocal arguments, and also extended their result to the Dirac--Coulomb system.

	We also want to emphasize that we are not aware of analogous late-time tail results for \eqref{CSF}. In setting of \eqref{ISP}, the main novelty of the present work are the late-time tails for higher spherical harmonics $\phi_{\geq \ell}$ in Theorem~\ref{main.theorem}.
	
	Beyond the context of scale-invariant wave equations, the study of sharp time-decay for wave equations with potential decaying faster than $O(r^{-2})$ has a long history, see \cite{wave.Schlag} and references therein. See also the more recent results on late-time tails in \cite{Faster1,Faster2,Faster3}.

	\subsection{Late-time tails on black hole spacetimes}\label{NP.intro}
	The numerology of decay rates in late-time tails for uncharged wave equations on Schwarzschild black hole backgrounds dates back to \cite{Pricepaper} and is known as \emph{Price's law}. The existence and properties of these tails\footnote{Decay upper bounds consistent with Price's law have in fact been obtained earlier, see \cite{Tataru,Tataru2,Schlag1,Schlag2,PriceLaw}.} were first confirmed in a mathematically rigorous setting in \cite{AAG1} and have since been extended to higher spherical harmonics and to the setting of the non-spherically symmetric sub-extremal Kerr black hole spacetimes \cite{Hintz, AAG3, AAG2} and to extremal Reissner--Nordstr\"om spacetimes \cite{Dejan_extreme}.
	
	In \cite{inverse.Dejan}, the first author considered \eqref{ISP} on Reissner--Nordstr\"{o}m black holes and established the existence of late-time tails with decay rates similar to those appearing in Theorem~\ref{main.theorem}. While this setting lacks the singularity properties at $r=0$ of \eqref{ISP} on Minkowski, the metric has a more complicated conformal structure, and features trapped null geodesics as well as a degeneracy of conserved energies at the event horizon, so the methods of the present paper do not apply.
	
	In subsequent work of the first author \cite{Gaj23}, late-time tails were derived on extremal Kerr spacetimes, where it was shown that they are intimately connected to the existence of instabilities. The decay rates in the extremal Kerr setting feature a numerology that strongly resembles the numerology in Theorem~\ref{main.theorem} and \cite{inverse.Dejan}, with the black hole event horizon playing the role of null infinity.
	
	\subsection{Key ideas of the proof}\label{section.key}
	Our novel geometric method relies on three key ideas as follows, which turn the problem of late-time asymptotics into  an elementary well-posedness argument.
	\begin{enumerate}[(A)]
		\item Turning the global problem of asymptotics of \eqref{ISP}/\eqref{CSF} on Minkowski spacetime into a local a larger spacetime by embedding Minkowski spacetime into a local patch of $AdS_2\times \s^2$.
		\item Show \emph{local} well-posedness for the wave equation  on the  $AdS_2\times \s^2$ spacetime under regular boundary conditions at the $AdS_2\times \s^2$ timelike conformal boundary, see Figure~\ref{fig:penrose1}.
		\item Relate the solution to the wave equation on  $AdS_2\times \s^2$ to the solution of \eqref{ISP}/\eqref{CSF} and show that the local-in-time boundedness of the former is equivalent to sharp global late-time asymptotics for the latter.

	\end{enumerate}
	
	\subsubsection{(A)  Global $\to$ local: embedding Minkowski into a local patch of $AdS_2\times \s^2$}
	
	We start by considering the Minkowski spacetime $({\rm Mink}^{3+1},m)$. Letting $r$ denote the standard area radius function, we conformally rescale:
	\begin{equation*}
		\widehat{m}=\frac{1}{r^2}m=-r^{-2} dt^2 + r^{-2} dr^2 + d\theta^2 + \sin^2(\theta) d\varphi^2 .
	\end{equation*}
	and defining $x=\frac{1}{r}$, we observe that $({\rm Mink}^{3+1},\widehat{m})$ can be isometrically embedded into the so-called Bertotti--Robinson spacetime $(AdS_2\times \s^2,\widehat{m})$ whose metric is given by  \begin{equation*}
		g_{AdS_2\times \s^2}= -x^2  dt^2 +   x^{-2}dx^2+d\theta^2 + \sin^2(\theta) d\varphi^2.
	\end{equation*} 
	As is clear from Figure~\ref{fig:penrose1}, understanding asymptotic properties towards $i^+$ in ${\rm Mink}^{3+1}$ is related to understanding asymptotic properties towards the $AdS_2\times \s^2$ boundary at $x=\infty$ (recalling $x=\frac{1}{r}$).
	
	Furthermore, as $AdS_2\times \s^2$ is globally time symmetric, we can embed another copy of $({\rm Mink}^{3+1},\widehat{m})$ by translating $({\rm Mink}^{3+1},\widehat{m})$ in the direction of the global timelike Killing vector field on $AdS_2\times \s^2$, see Figure~\ref{fig:penrose2} (in red).	In particular, we can translate  $({\rm Mink}^{3+1},\hat{m})$ to obtain $(\widetilde{\rm Mink}^{3+1},\widehat{m})$, which  overlaps with $({\rm Mink}^{3+1},\hat{m})$ (see Figure~\ref{fig:penrose2}) and has the property that $i^+$ in $({\rm Mink}^{3+1},\widehat{m})$ coincides with the origin $R=0$ at time $T=0$ in standard spherical coordinates $(T,R,\theta,\varphi)$ on $\widetilde{\rm Mink}^{3+1}$. 
	
	Denoting the Minkowski metric on ${\rm Mink}^{3+1}$ as follows:
	\begin{equation*}
		\widetilde{m}=R^2 \widehat{m},
	\end{equation*}
	we moreover have the relation $\widetilde{m}=I^* m$ in the intersection $\widetilde{\rm Mink}^{3+1}\cap {\rm Mink}^{3+1}$, where $I$ denotes the inversion map and $I^*$ its pushforward, see for example \cite{christo.book}[\S4.1]. 
	
	In this paper, we will restrict to $J^+(\Sigma_{-1})\subset  {\rm Mink}^{3+1}$, where $\Sigma_{-1}$ is a hypersurface of hyperboloidal nature in ${\rm Mink}^{3+1}$ that coincides with $\{T=-1\}\cap\{R<1\}\subset \widetilde{\rm Mink}^{3+1}$ so that $J^+(\Sigma_{-1})$ is entirely contained in $\widetilde{\rm Mink}^{3+1}$.
	
	Let $\phi$ denote a solution to \eqref{CSF} with $D, A, F$ given by \eqref{def.gauge.deriv}, \eqref{def.pot}, \eqref{Maxwell.static}. 
	Then $\widetilde{F}=\frac{e}{R^2}dT\wedge dR=F$ and moreover, $\widetilde{\phi}=\frac{r}{R}\phi$ satisfies
	\begin{equation*}
		\widetilde{m}^{\mu\nu}\widetilde{D}_{\mu}\widetilde{D}_{\nu}\widetilde{\phi}=0,
	\end{equation*}
	with $\widetilde{D}_{\mu}=\widetilde{\partial}_{\mu}+i q e \widetilde{A}_{\mu}$ for some $\widetilde{A}$ with $d\widetilde{A}=F$. In view of the discussion above and Figure~\ref{fig:penrose2}, we have related the problem of determining the \emph{global} late-time asymptotics of $\phi$ to the problem of determining the \emph{local} asymptotics of $\widetilde{\phi}$ in $J^+(\Sigma_{-1})\cap \widetilde{\rm Mink}^{3+1}$ in a neighborhood of $(T,R)=(0,0)$.
	\begin{figure}
		\begin{center}
			\includegraphics[scale=0.4]{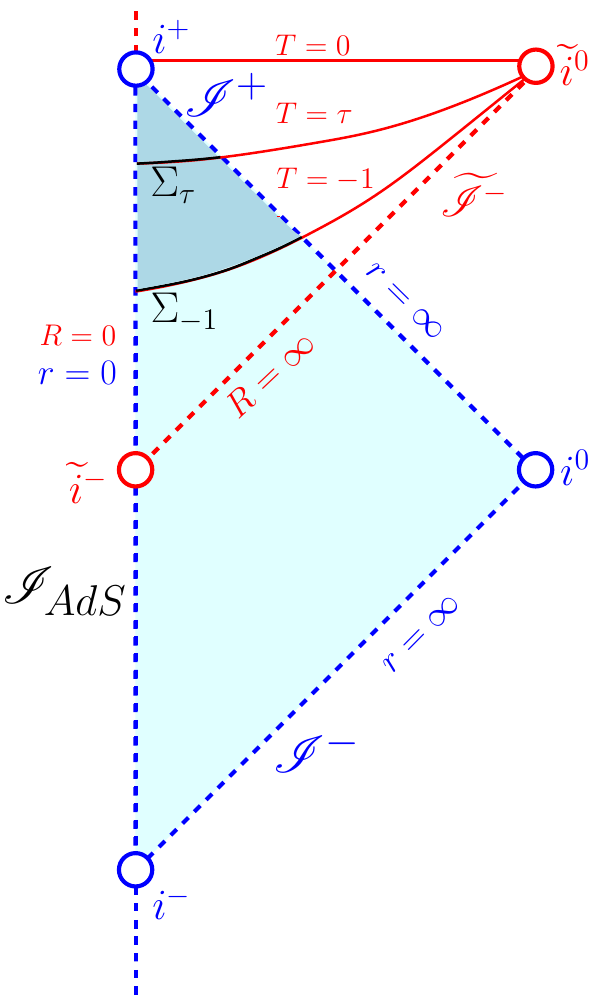}
		\end{center}
		\caption{A 2-dimensional representation of two copies of Minkowski spacetime  embedded in $AdS_2\times \s^2$.}
		\label{fig:penrose2}
	\end{figure}

	\subsubsection{(B) Asymptotic behaviour at the $AdS_2\times \s^2$ conformal boundary}
	To understand how $\widetilde{\phi}$ behaves as $(T,R)\to (0,0)$, we need to establish first of all a suitable wellposedness theory with suitable boundary conditions and in appropriate energy spaces. The wellposedness theory follows from applying ideas developed in \cite{WarnickAdS} on asymptotically anti de Sitter spacetimes, which are relevant because $(T,R)=(0,0)$ can be interpreted as a sphere on the timelike boundary at infinity in $AdS_2\times \s^2$, as can be seen comparing Figure~\ref{fig:penrose1} and Figure~\ref{fig:penrose2}.
	
	Heuristically, one would expect the following asymptotic behavior approaching the $AdS_2$ boundary at $R=0$ (or $\frac{1}{R}\to \infty$):
	\begin{equation*}
		r	\tilde{\phi}_{\ell} \sim A R^{\frac{1}{2}+i qe+\alpha_{\ell}}+BR^{\frac{1}{2}+i qe-\alpha_{\ell}},
	\end{equation*}
	with
	\begin{align*}
		\alpha_{\ell}=&\:\sqrt{\frac{1}{4}+\ell(\ell+1)+a}\quad \textnormal{for the equation \eqref{ISP}},\\
		\alpha_{\ell}=&\:\sqrt{\frac{1}{4}+\ell(\ell+1)-e^2q^2}\quad \textnormal{for the equation \eqref{CSF}}.
	\end{align*}
	The choice of ``Dirichlet boundary conditions'' would correspond to restricting to $B=0$.

	The relevance of the exponents $\alpha_{\ell}$ can be seen by employing $\alpha$-twisted derivatives:
	\begin{equation*}
		\partial_R^{\alpha}(\cdot)=R^{-\alpha}\partial_R(R^{\alpha}(\cdot)).
	\end{equation*}
	
	If we rescale $\Phi=R^{-\frac{1}{2}-iqe}(r\widetilde{\phi})$ and we make the choice $A=\frac{e}{R}(dT-dR)$, \eqref{ISP} reduces to
	\begin{equation}\label{ISP2}
		0=-\partial_T^2   \Phi+\partial_R^{1-\alpha_0} \partial_R^{\alpha_0}  \Phi+R^{-2} \sD_{\mathbb{S}^2} \Phi.
	\end{equation}
	and \eqref{CSF} reduces to:
	\begin{equation}\label{CSF2}
		0=-\partial_T^2   \Phi+\partial_R^{1-\alpha_0} \partial_R^{\alpha_0}  \Phi+R^{-2} \sD_{\mathbb{S}^2} \Phi - \frac{2i q e\   \partial_T \Phi}{R}.
	\end{equation}
	In view of the above equations, it is straightforward to obtain energy estimates for $\Phi$ with respect to the following energy, which is natural while working on $AdS_{2} \times \s^2$ (see also \cite{WarnickAdS}):
	\begin{equation*}
		\int_{\Sigma_T} \left[ |\partial^{\alpha_0}_R\Phi|^2+|\partial_T\Phi|^2+|\snabla \Phi|^2\right]\,RdR,
	\end{equation*} 
	which allows  use to obtain existence and uniqueness of solutions with Dirichlet boundary conditions.
	
	Employing moreover standard elliptic estimates and requiring more regularity of the initial data, we can control the second-order quantity
	\begin{equation*}
		\int_{\Sigma_T}|\partial_R^{1-\alpha_0}\partial_R^{\alpha_0}\Phi|^2\,RdR,
	\end{equation*}
	which we can use to obtain a uniform bound 	in terms of initial data on
	\begin{equation}\label{Linfinity.intro}
		\|R^{-\alpha_0}\Phi\|_{L^{\infty}(\Sigma_T)}
	\end{equation}
	and moreover conclude that $R^{-\alpha_0}\Phi$ attains a finite limit at $R=0$, as expected from the heuristic asymptotics and the choice of Dirichlet boundary conditions. In the next subsection~\ref{nextsub}, we will explain how the control of \eqref{Linfinity.intro} provides the sharp asymptotics claimed in Theorem~\ref{main.theoremell0}.
	
	To show that moreover $R^{-\alpha_0}\Phi_{\geq 1}=0$ at $R=0$, and more generally, to show that $R^{-\alpha_{\ell}}\Phi_{\geq \ell}$ attains a finite limit at $R=0$, we instead rearrange the equation for $\Phi$ to obtain an equation involving $\alpha_{\ell}$ instead of $\alpha_0$ (in the next subsection~\ref{nextsub} we will explain the relation to Theorem~\ref{main.theorem}):
	\begin{equation*}
		0=-\partial_T^2   \Phi_{\geq \ell}+\partial_R^{1-\alpha_{\ell}} \partial_R^{\alpha_{\ell}}  \Phi_{\geq \ell}+R^{-2} \left[\sD_{\mathbb{S}^2}+\ell(\ell+1)\right] \Phi_{\geq \ell} - \frac{2i q e\   \partial_T \Phi_{\geq \ell}}{R},
	\end{equation*}
	where the last term on the RHS is only present for \eqref{CSF}.
	
	We can derive improved elliptic estimates near $R=0$ by multiplying the above equation by negative weights in $R$ and we obtain in particular control over:
	\begin{equation*}
		\int_{\Sigma_T}R^{2-2\alpha_{\ell}}|\partial_R^{1-\alpha_{\ell}}\partial_R^{\alpha_{\ell}}\Phi_{\ell}|^2\,RdR,
	\end{equation*}
	and derive boundedness of $||R^{-\alpha_{\ell}}\Phi_{\geq \ell}||_{L^{\infty}(\Sigma_T)}$ in terms of initial data, after applying $R$-weighted elliptic estimates hierarchically. The use of $R$-weighted hierarchies of elliptic estimates plays an important role in \cite{AAG1,AAG2,AAG3,QNM}.
	
	To show that the limit $R^{-\alpha_{\ell}}\Phi_{\geq \ell}$ is generically \emph{non-zero} at $(T,R)=(0,0)$, we simply use that
	\begin{equation*}
		R^{\alpha_{\ell}}Y_{\ell m}(\theta,\varphi)
	\end{equation*}
	constitute a $2\ell+1$ dimensional set of solutions to  \eqref{ISP2}/\eqref{CSF2}. By adding these solutions to our initial data, we can show that the $R^{-\alpha_{\ell}}\Phi_{\geq \ell}$ is zero at $(T,R)=(0,0)$ for a subset of initial data of at least codimension $2\ell+1$ within a suitably weighted Hilbert space of initial data.
	\subsubsection{(C) From local radial asymptotics near origin to late-time tails}\label{nextsub}

	The conclusion of the proof relies on an explicit computation:	The leading-order asymptotics  in $R$ 
	\begin{align*}
		R^{-\alpha_{\ell}}\Phi_{\geq \ell}(T,R,\theta,\varphi)=&\:\sum_{m=-\ell}^{\ell} P_{\ell m}(T)Y_{\ell m}+O(R^{\nu})\\
		=&\:\sum_{m=-\ell}^{\ell} P_{\ell m}(0)Y_{\ell m}+O(R^{\nu})+O(|T|)
	\end{align*}
	for some $\nu>0$, imply that in double null coordinates $u=\frac{1}{2}(t-r)$ and $v=\frac{1}{2}(t+r)$, using that $R(u,v)=\frac{r}{u v}$ and $|T|(u,v)=\frac{1}{u}+\frac{1}{v}$,
	\begin{equation*}
		\begin{split}
			r\phi_{\geq \ell}(u,v)=&\:R(u,v)^{\frac{1}{2}+i q e+\alpha_{\ell}}\sum_{m=-\ell}^{\ell} P_{\ell m}(0)Y_{\ell m}+O(R^{\frac{1}{2}+\alpha_{\ell} +\nu})+R^{\frac{1}{2}+\alpha_{\ell} } O(|T|)\\
			=&\: \left(\frac{v-u}{uv}\right)^{\frac{1}{2}+i q e}+\alpha_{\ell}\sum_{m=-\ell}^{\ell} P_{\ell m}(0)Y_{\ell m}+O\left(\left(\frac{v-u}{uv}\right)^{\frac{1}{2}+\alpha_{\ell}+\nu}\right)+\left(\frac{v-u}{uv}\right)^{\frac{1}{2}+\alpha_{\ell}}O\left(\frac{1}{u}+\frac{1}{v}\right)
		\end{split}
	\end{equation*}
	with the gauge choice $A=\frac{2e}{r} du- 2e d ( \ln u)$ (namely \eqref{gauge}) for \eqref{CSF} and $e=0$ for \eqref{ISP}.

	Now we can obtain the following global leading-order late-time asymptotics:
	\begin{equation*}
		r\phi_{\geq \ell}(u,v)=\left(\frac{v-u}{uv}\right)^{\frac{1}{2}-i q e+\alpha_{\ell}}\sum_{m=-\ell}^{\ell} P_{\ell m}(0)Y_{\ell m}+\left(\frac{v-u}{uv}\right)^{\frac{1}{2}+\alpha_{\ell}}O(u^{-\nu}).
	\end{equation*}

	\subsection{Acknowledgements}  DG acknowledges funding through the ERC Starting Grant 101115568 and MVdM gratefully acknowledge support from the NSF Grant	 DMS-2247376.

	\section{Geometric preliminaries and equations}\label{section.prelim}
	In this section, we introduce the main geometric notions and equations that play a role in the paper.
	\subsection{Double null  coordinates on Minkowski spacetime}
	
	We work with the Minkowski spacetime $(\mathcal{M},m)$, with $\mathcal{M}=\R\times (0,\infty)\times \s^2$, 
	 where $m$ is the Minkowski metric that takes the following form in standard spherical coordinates: $$ m= -dt^2+ dr^2+ r^2 d\sigma_{\mathbb{S}^2}= -4du dv + r^2 d\sigma_{\mathbb{S}^2},$$
	where $d\sigma_{\mathbb{S}^2}$ is the standard metric on $\mathbb{S}^2$ and the null coordinates $(u,v)$ are defined as follows:
	$$u= \frac{t-r}{2}, \hskip 5 mm  v = \frac{t+r}{2}=u+r  .$$

	We introduce also timelike infinity, future null infinity and the center of spherical symmetry, which can formally be defined as follows:
	\begin{equation} \label{hyp.oldcor}
		i^+:=\{u=\infty,\ v=\infty\}, \hskip 5 mm \mathcal{I}^+:=\{v=\infty\}, \hskip 5 mm \Gamma:=\{r=0\}.
	\end{equation}

	\subsection{Conformal embedding and coordinates on $AdS_2 \times \mathbb{S}^2$} \label{section.conformal}
	Consider the conformally rescaled Minkowski metric:
	\begin{equation*}
		\hat{m}= \frac{1}{r^2}(- dt^2+ dr^2)+  d\sigma_{\mathbb{S}^2}.
	\end{equation*}
	We can interpret the corresponding spacetime $(\mathcal{M},\hat{m})$ as a subset of $AdS_2 \times \mathbb{S}^2$. Indeed, defining
	\begin{equation*}
		x=\frac{1}{r}
	\end{equation*}
	gives $\hat{m}=-x^2dt^2+\frac{1}{x^2}dx^2$, which is the $AdS_2 \times \mathbb{S}^2$ metric in Poincar\'e coordinates. We can extend  $(\mathcal{M},\hat{m})$ to obtain the full spacetime $AdS_2 \times \mathbb{S}^2$, by considering the transformation from $(t,x)$ to $(\tau,\rho)$ coordinates:
	\begin{align*}
		x=&\: \cosh \rho\cos\tau+\sinh\rho,\\
		t=&\: \frac{\cosh\rho \sin t}{\cosh^2\rho \sin^2\tau-1}(\sinh \rho-\cosh\rho \cos \tau).
	\end{align*}
	Then we can express
	\begin{equation*}
		\hat{m}=-\cosh^2\rho\, d\tau^2+d\rho^2+d\sigma_{\s^2}
	\end{equation*}
	and $\mathcal{M}=\{(\tau,\rho)\in (0,\pi)\times \R\,|\,-\frac{\pi}{2}<2\arctan({\tanh{\frac{\rho}{2}}})<-|\tau-\frac{\pi}{2}|\}\subset \R\times \R\times \s^2$.
	The conformal boundary $\mathscr{I}_{AdS}$ of $AdS_2 \times \mathbb{S}^2$ corresponds to the union $\{\rho_*=-\frac{\pi}{2}\}\cup \{\rho_*=\frac{\pi}{2}\}$, where $\rho_*=2\arctan({\tanh{\frac{\rho}{2}}})$.

	In the remainder of the paper, we will restrict to the subset $\mathcal{M}_{u>0}:=\{u>0\}\subset \mathcal{M}$, where we can define the change of coordinate map $$ (u,v) \in \R^+ \times \R^+ \mapsto (U,V)= (-\frac{1}{u},-\frac{1}{v})\in \R^- \times \R^-.$$ 
	
	Then, we define the function $R: \mathcal{M}_{u>0} \to\R^+$: \begin{equation}\label{R.def}
		R(u,v):= \frac{r}{uv}=\frac{r}{u(u+r)}=\frac{4r}{t^2-r^2}=\frac{1}{u}-\frac{1}{v}= V-U \in \R^+.
	\end{equation}	It follows immediately from the above that we can express $\hat{m}$ as follows on $\mathcal{M}_{u>0}$:
	\begin{equation}
		\hat{m}=  -4R^{-2} dU dV+ d\sigma_{\mathbb{S}^2}.
	\end{equation}
	We will also introduce the function $T(U,V)= U+V \in \R^{-}$. An elementary computation shows that in $(T,R)$ coordinates, we can express (we also note  that the coordinate transformation $(t,r,\theta,\varphi) \rightarrow (T,R,\theta,\varphi)$ is a conformal isometry):
	\begin{equation}
		\hat{m}= R^{-2} [ -(dT)^2+ (dR)^2] + d\sigma_{\mathbb{S}^2}.
	\end{equation}
	Denoting $\widetilde{m}=R^2\hat{m}=\frac{R^2}{r^2}m$ and observing that $\widetilde{m}$ is another Minkowski metric, we can therefore extend $(\mathcal{M}_{u>0},\widetilde{m})$ to another Minkowski spacetime $(\widetilde{\mathcal{M}},\widetilde{m})$, with $\widetilde{\mathcal{M}}=\R_T\times (0,\infty)_R\times \s^2$:
	\begin{equation*}
		\mathcal{M}_{u>0}=\{U<0,V<0\}\subset \widetilde{\mathcal{M}}.
	\end{equation*}

	We denote the level sets of $T$ in $\widetilde{\mathcal{M}}$ as follows: $\Sigma_T= \{T\} \times (0,\infty)_R \times \s^2 $. Note that $$\Sigma_{-1}\cap \mathcal{M}_{u>0}= \left\{ \frac{1}{u}+\frac{1}{v} =1\right\} \times \s^2= \{ (t-2)^2 -r^2 =4\} \subset \mathcal{M},$$ which is consistent with Section~\ref{section.precise}.
	
	We can also characterize:
	\begin{equation*}
		\Sigma_{-1}\cap \mathcal{M}_{u>0}= \Sigma_{-1}\cap \{R<1\},
	\end{equation*}
	from which it follows that for $T\in [-1,0)$:
	\begin{equation*}
		\Sigma_{T}\cap \mathcal{M}_{u>0}\subset  \widetilde{\mathcal{M}}\cap \{R<1\}.
	\end{equation*}
	
	Note that, recalling the definitions from \eqref{hyp.oldcor}, we have that $i^+$ and $\mathcal{I}^+$ can be interpreted as a point and a null hypersurface in $\widetilde{M}$ respectively, with
	\begin{equation} i^+=\{ (U,V)=(0,0)\}, \hskip 5 mm
		\mathcal{I}^+= \{ V =0,\, -1<U<0\}, \hskip 5 mm \Gamma = \{ R=0,\, -1<T<0\}=\{V=U,\, -1<U<0\}.
	\end{equation} 
	
	We will work with the new coordinates $(T,R,\theta,\varphi)$ on $\widetilde{\mathcal{M}}$ and we use this coordinate system to express the Morawetz vector field $K$ on the original Minkowski spacetime $(\mathcal{M},m)$ and its Lorentzian orthogonal $K^{\perp}$, first introduced in \eqref{K.def},\eqref{Kperp.def}: \begin{align} 
		&K= \partial_T= \partial_V + \partial_U= v^2 \partial_v+ u^2 \partial_u ,
		&	K^{\perp}= \partial_R= \partial_V- \partial_U = v^2 \partial_v - u^2 \partial_u .
	\end{align}

	\subsection{Charged wave equation and gauge choices} \label{section.CSF}
	
	A key property of the system of equations \eqref{CSF} is its invariance under the following gauge transformation \begin{align*}
		&\tilde{\phi}=  e^{-iq f} \phi,\\ & \tilde{A}= A+ df,
	\end{align*} leaving a \emph{gauge freedom} in the choice of $A$.  In $(u,v,\theta,\varphi)$ coordinates, we will take advantage of this gauge freedom and impose the following conditions (see \cite{Moi2}) \begin{equation} \label{gauge.Av}
		A_v \equiv 0,\ A_{\theta}=A_{\varphi}=0.
	\end{equation} Note that the second condition simply states that $A$ is spherically symmetric. Assuming \eqref{gauge.Av}, $dA=F$ becomes \begin{equation} \label{maxwell}
		\partial_v A_u = -\frac{2e}{r^2}.
	\end{equation}
	
	Note that \eqref{gauge.Av} does not fix completely the gauge freedom, as $\tilde{A}= A+ d ( f(u))$ also satisfies \eqref{maxwell} for any $f$. Introducing the radiation field $\psi:= r\phi$ and under \eqref{gauge.Av},
	\eqref{CSF} becomes \begin{equation*}
		\partial_u \partial_v \psi  + iq  A_u  \partial_v \psi - \frac{i (q e )\psi }{r^2} = r^{-2} \sD_{\mathbb{S}^2} \psi.
	\end{equation*}  We can now express the above in $(U,V,\theta,\varphi)$ coordinates: \begin{equation} \label{CSF.1}
		\partial_U \partial_V \psi  + iq  A_U  \partial_V \psi -  \frac{i q e\ \psi }{R^2} = R^{-2} \sD_{\mathbb{S}^2} \psi.
	\end{equation}

	Next, we will fix our remaining gauge freedom; note that, since $U=\frac{-1}{u}$ is a bijection, fixing $A_u$ is equivalent to fixing $A_U$. In addition to \eqref{gauge.Av}, we then impose \begin{equation}\label{gauge.Au}
		A_U = \frac{2e}{R}.
	\end{equation}	
	
	Expressing $A$ and $F$ in $(U,V)$ coordinates, $dA=F$ is  indeed satisfied. Then \eqref{CSF.1} becomes \begin{equation} \label{CSF.2}
		\partial_U \partial_V \psi  + \frac{2iq e\   \partial_V \psi}{R} -  \frac{i q e\ \psi }{R^2} = R^{-2} \sD_{\mathbb{S}^2} \psi.
	\end{equation}
	
	Then finally we express \eqref{CSF.2} in $(T,R)$ coordinates \begin{equation} \label{CSF.3}
		\partial_T^2   \psi - \partial_R^2   \psi  + \frac{2iq e\   (\partial_T \psi+\partial_R \psi)}{R} -  \frac{i q e\ \psi }{R^2} = R^{-2} \sD_{\mathbb{S}^2} \psi.
	\end{equation}

	\subsection{Renormalizations of the scalar field and twisted derivatives}
	
	Next we will introduce a renormalization of $\psi$ that will be more convenient for deriving energy estimates and local wellposedness:
	\begin{equation} \label{Phi.def}
		\Phi:=  R^{-\frac{1}{2}- iqe } \psi.
	\end{equation} We will also introduce the following numbers  for all $\ell \in \mathbb{N}_0$: \begin{align}
		\label{alpha.def}\alpha_{\ell} = \frac{\sqrt{1-4 (qe)^2 + 4 \ell (\ell+1)}}{2} \in (\sqrt{\ell(\ell+1)},\frac{\sqrt{1 + 4 \ell (\ell+1)}}{2})
	\end{align}
	
	Note in particular that \begin{align*}
		&0 <\alpha_0 <\frac{1}{2},\\ & \alpha_{\ell} \geq \sqrt{2}>1, \text{ for all } \ell \geq 1.
	\end{align*} 
	
	An straightforward computation then shows that \eqref{CSF.3} is equivalent to  \begin{equation} \label{CSF.4}
		\partial_T^2   \Phi - \partial_R^{1-\alpha_0} \partial_R^{\alpha_0}  \Phi  + \frac{2i q e\   \partial_T \Phi}{R}= R^{-2} \sD_{\mathbb{S}^2} \Phi,
	\end{equation}
	where $\partial_R^{1-\alpha_0}$ and $\partial_R^{\alpha_0}$ denote \emph{$\alpha$-twisted derivatives}, which are defined as follows:
	
	\begin{defn} \label{def.twisted}
		Let $\alpha\in \R$. We define the \emph{$\alpha$-twisted derivative} $\partial_R^{\alpha}$ as the following differential operator:
		\begin{equation*}
			\partial_R^{\alpha}(\cdot)=R^{-\alpha}\partial_R(R^{\alpha} (\cdot)).
		\end{equation*}
	\end{defn} 
	Twisted derivatives were introduced in the study of wave equations on asymptotically anti de Sitter spacetimes in \cite{WarnickAdS}.
	
	We moreover denote with $\snabla_{\s^2}$ the standard covariant derivative on the unit round sphere.
	
	It will moreover be convenient to rewrite the above equation for $\Phi$ more generally as an equation involving twisted derivatives with respect to $\alpha_{\ell}$ instead of $\alpha_0$:
	\begin{equation} \label{CSF.4.5}
		\partial_T^2   \Phi - \partial_R^{1-\alpha_{\ell}} \partial_R^{\alpha_{\ell}}  \Phi  + \frac{2i q e\   \partial_T \Phi}{R}= R^{-2}[ \ell (\ell+1) \Phi+ \sD_{\mathbb{S}^2} \Phi],
	\end{equation}
	where we used the identity
	\begin{equation}
		\label{eq:relalphalalpha0}
		\partial_R^{1-\alpha_{\ell}} \partial_R^{\alpha_{\ell}}=\partial_R^{1-\alpha_{0}} \partial_R^{\alpha_{0}}-R^{-2}\ell(\ell+1).
	\end{equation}

	\subsection{Wave equation with an inverse-square potential}
	
	For any $a >-\frac{1}{4}$, we consider the equation with a fixed Coulomb potential $\frac{a}{r^2}$  \begin{equation} \label{wave.pot}
		m^{\mu \nu} \nabla_{\mu} \nabla_{\nu} \phi = \frac{a \phi}{r^2}.
	\end{equation}

	Then, under the same coordinates as in Section~\ref{section.CSF}, defining $\psi = r\phi$, we have that:
	\begin{equation} \label{wave.pot2}
		\partial_T^2 \psi - \partial_R^{2} \psi + \frac{a}{R^2} \psi = R^{-2} \sD_{\mathbb{S}^2} \psi.
	\end{equation} 
	After introducing the following quantities:
	\begin{align*}
		\Phi= &\:R^{-\frac{1}{2}} \psi, \\ 
		\alpha_0 =&\: \frac{\sqrt{1+4a}}{2} >0, \\ 
		\alpha_\ell =&\: \frac{\sqrt{1+4a+ 4\ell (\ell+1)}}{2} >\sqrt{\ell (\ell+1)},\\  
		\tpsi =&\: R^{-\alpha_0} \Phi,
	\end{align*}
	we obtain the equations:
	\begin{equation} \label{wave.pot3}
		\partial_T^2   \Phi - \partial_R^{1-\alpha_0} \partial_R^{\alpha_0} \Phi = R^{-2} \sD_{\mathbb{S}^2} \Phi.
	\end{equation} \begin{equation} \label{wave.pot3.5}
		\partial_T^2   \Phi - \partial_R^{1-\alpha_{\ell}} \partial_R^{\alpha_{\ell}} \Phi = R^{-2} [\ell(\ell+1)\Phi+\sD_{\mathbb{S}^2} \Phi].
	\end{equation}   
	
	\textbf{Observe that \eqref{wave.pot3}
		is of the same form as \eqref{CSF.4}
		with $qe=0$} (albeit for a different choice of $\alpha_0$, $\alpha_{\ell}$).  In the particular case  $\alpha_0 \in (0,1)$, using that $\alpha_1>1$, we will be able to apply the estimates for \eqref{CSF.4}
	directly to the setting of \eqref{wave.pot2}.

	\subsection{Spherical harmonic modes} \label{spherical.section}
	In this section, we review basic properties involving decompositions into spherical harmonic modes.
	\begin{defn}
		We define $(Y_{\ell,m})_{\ell \in \mathbb{N} \cup \{0\}, m \in [-\ell,\ell]}$ the orthonormal base of eigenfunctions on $\mathbb{S}^2$ for $\slashed{\Delta}_{\mathbb{S}^2}$ such that $$  \slashed{\Delta}_{\mathbb{S}^2} Y_{\ell,m} =-\ell (\ell+1) Y_{\ell,m}.$$
	\end{defn}
	We now introduce the projection operators: \begin{defn}\label{def.proj}
		Let $f \in L^2(\s^2)$. We define the projection on the angular mode of order $\ell \in \mathbb{N}\cup\{0\}$ as follows:
		\begin{align*}
			&\pi_{\ell} (\Phi)(\theta,\varphi) = \Phi_{\ell}(\theta,\varphi):= \sum_{m=-\ell}^{\ell}\left[\int_{ \mathbb{S}^2} (\Phi \cdot \overline{Y_{\ell,m}})(\theta',\varphi')\, \sin \theta'  d\theta' d\varphi'\right] Y_{\ell,m}(\theta,\varphi),\\ & \pi_{\geq \ell} (\Phi) = \Phi_{\geq \ell}= \sum_{k=\ell}^{+\infty} \Phi_k.
		\end{align*} 
	\end{defn}
	For the sake of convenience, we will frequently make use of the shorthand notation $\omega=(\theta,\varphi)$, with $\theta$ and $\varphi$ standard spherical coordinates and $d\sigma_{\s^2}=\sin \theta d\theta d\varphi$.

	\begin{lemma} \label{lem.poincare}
		Let $f \in L^2(\mathbb{S}^2)$ and decompose $f = \sum_{\ell =0}^{+\infty} f_{\ell}$. If we assume $f \in H^2(\mathbb{S}^2)$, then for all $\ell_0 \in \mathbb{N}_0$, the following identities hold:
		\begin{align} 
			\label{poinc1} 
			\int_{\s^2}[|\snabla_{\s^2}f_{\ell_0}|^2-
			\ell_0 (\ell_0+1)|f_{\ell_0}|^2]\,d\sigma_{\s^2} = &\:\int_{\s^2}|(\slashed{\Delta}_{\s^2}+\ell_0(\ell_0+1))f_{\ell_0}|^2\,d\sigma_{\s^2}\\ \nonumber
			=&\:0, \\ \nonumber
			(\ell(\ell+1)-\ell_0(\ell_0+1))\int_{\s^2}[|\snabla_{\s^2}f_{ \ell}|^2-\ell_0 (\ell_0+1)|f_{ \ell}|^2]\,d\sigma_{\s^2} =&\:\int_{\s^2} |(\slashed{\Delta}_{\s^2}+ \ell_0(\ell_0+1))f_{\ell}|^2\,d\sigma_{\s^2}\\
			= [ \ell (\ell+1)-\ell_0(\ell_0+1)]^2 \int_{\s^2}|f_{ \ell}|^2\,d\sigma_{\s^2} & \text{ for any } \ell \geq \ell_0 .
			\label{poinc2} \\ 
			\int_{\s^2}[|\snabla_{\s^2}f_{ \ell}|^2-\ell_0(\ell_0+1)|f_{ \ell}|^2]\,d\sigma_{\s^2}= \frac{\ell (\ell+1)-\ell_0(\ell_0+1)}{\ell (\ell+1)}& \int_{\s^2}|\snabla_{\mathbb{S}^2}f_{ \ell}|^2\,d\sigma_{\s^2} \text{ for any } \ell \geq \ell_0+1.  \label{poinc2.5} 
		\end{align} Furthermore, the following inequalities hold for $f_{\geq \ell}$: 
		\begin{align} \label{poinc3}
			\int_{\s^2}[|\snabla_{\s^2}f_{ \geq\ell } |^2-\ell(\ell+1)|f_{ \geq\ell  } |^2]\,d\sigma_{\s^2} \leq &\:  \frac{1}{2 (\ell+1)}\int_{\s^2} |(\slashed{\Delta}_{\s^2}+2)f_{\geq  \ell } |^2\,d\sigma_{\s^2},\\
			\label{poinc4}
			\int_{\s^2}|\snabla_{\s^2}f_{ \geq \ell } |^2\,d\sigma_{\s^2} \leq  &\:\frac{\ell+2}{2} \int_{\s^2}[|\snabla_{\s^2}f_{ \geq \ell } |^2 - \ell(\ell+1)|f_{ \geq \ell } |^2]\ d\sigma_{\s^2}.
		\end{align} 
	\end{lemma}
	\begin{proof}
		We have by linearity that $\slashed{\Delta}_{\mathbb{S}^2} f_{\ell} =-\ell (\ell+1) f_{\ell}$. Multiplying by $\bar{f}$ and integrating by parts over $\s^2$ gives $$ \int_{\s^2} |\snabla_{\mathbb{S}^2} f_{\ell}|^2 d\sigma_{\s^2}=\ell (\ell+1) \int_{\s^2} |f_{\ell}|^2 d\sigma_{\s^2}.$$ Thus, for $\ell=\ell_0$, \eqref{poinc1} is clear. For \eqref{poinc2}, simply note that $|(\slashed{\Delta}_{\mathbb{S}^2}+\ell_0(\ell_0+1)) f_{\ell}|^2 =(\ell (\ell+1) -\ell_0 (\ell_0+1))^2 |f_{\ell}|^2$ and $$ \int_{\s^2} (|\snabla_{\mathbb{S}^2} f_{\ell}|^2-\ell_0 (\ell_0+1)|f_{\ell}|^2) d\sigma_{\s^2}=(\ell (\ell+1) -\ell_0 (\ell_0+1))\int_{\s^2} |f_{\ell}|^2 d\sigma_{\s^2}.$$ 
		
		Finally, \eqref{poinc3} and \eqref{poinc4} are obtained by summation, using the orthogonality of $Y_{\ell, m}$ for different values of $\ell$.
	\end{proof}
	
	\begin{rmk} \label{projsolution.rmk}
		Note that if $\Phi$ is a solution of \eqref{CSF.4} and thus also a solution of \eqref{CSF.4.5}, then both $\Phi_{\ell'}$ and $\Phi_{\geq \ell'}$ are also solutions of \eqref{CSF.4.5} for all $\ell' \in \mathbb{N}_0$. We will specifically consider the case $\ell'=\ell$ in the sequel.
	\end{rmk}

	\subsection{Function spaces: $\alpha$-twisted Sobolev spaces}
	\label{sec:functionspaces}
	Before we define the relevant function spaces, we observe the following integration-by-parts property with respect to $\alpha$-twisted derivatives: for any $T \in \R$ and $u,v\in C_c^{\infty}(\Sigma_T)$, we have that
	\begin{equation*}
		\int_{\Sigma_T}(\partial_R^{\alpha}u) \overline{v}\,Rd \sigma_{\mathbb{S}^2} dR=-\int_{\Sigma_T}u (\partial_R^{1-\alpha} \overline{v})\,Rd\sigma_{\mathbb{S}^2} dR.
	\end{equation*}

	In order to determine the natural Hilbert spaces for \eqref{CSF.4}, we assume $v\in C_c^{\infty}(\Sigma_T)$, we multiply both sides of \eqref{CSF.4} with $R\overline{  \partial_T\Phi}$, take the real part and integrate by parts (ignoring boundary terms when integrating by parts with respect to $\alpha$-twisted derivatives in the $R$ direction) to obtain the following energy identity:
	\begin{equation*}
		0=\frac{1}{2}\partial_T \int_{\Sigma_{T}} [R(\partial_ T \Phi )^2+R |\partial_R ^{\alpha_0}\Phi|^2+R^{-1}|\snabla_{\s^2}\Phi|^2]\,d\sigma_{\mathbb{S}^2} dR=:\frac{1}{2}\frac{d}{dT}E[\Phi](T).
	\end{equation*}
	
	The above definition of the energy $E[\Phi](T)$ motivates the choice of Hilbert spaces below.
	
	\begin{defn}\label{def.Sobolev}
		We consider the spaces $\underline{L}^2(\Sigma_T)=\underline{L}^2(\Sigma_T;\mathbb{C})$, corresponding to norms defined as follows:
		\begin{equation*}
			||f|_{\underline{L}^2(\Sigma_T)}^2:=\int_{\Sigma_T} |f|^2\,Rd\sigma_{\mathbb{S}^2} dR.
		\end{equation*}
		We moreover define the inner product on $\underline{L}^2(\Sigma_T)$ as follows:
		\begin{equation*}
			\la f, g\ra_{\underline{L}^2(\Sigma_T)}:=\int_{\Sigma_T} f\cdot \overline{g}\,R d\sigma_{\mathbb{S}^2} dR.
		\end{equation*}
		Given $u\in \underline{L}^2(\Sigma_T)$, we say  $w\in \underline{L}^2(\Sigma_T)$ is the weak $\alpha$-twisted derivative of $u$, if for all $v\in C_c^{\infty}(\Sigma_T)$:
		\begin{equation*}
			\int_{\Sigma_T} w \cdot \overline{v}\,Rd\sigma_{\mathbb{S}^2} dR=-\int_{\Sigma} u \cdot \partial_R^{1-\alpha}\overline{v}\,Rd\sigma_{\mathbb{S}^2} dR.
		\end{equation*}
		Now we define $\alpha$-twisted Sobolev spaces $\underline{H}^1_{\alpha}(\Sigma_T)$ on $\Sigma_T$ as follows:
		\begin{align*}
			\underline{H}_{\alpha}^1(\Sigma_T):=&\:\left\{R^{-1}f\in \underline{L}^2(\Sigma_T)\,\big|\, \snabla f\in  \underline{L}^2(\Sigma_T),\: \partial^{\alpha}_R f\in  \underline{L}^2(\Sigma_T)\right\},\\
			\la f,g\ra_{\underline{H}^1_{\alpha}(\Sigma_T)}:=&\: \int_{\Sigma_T} R^{-2}f \overline{g} \,Rd\sigma_{\mathbb{S}^2} dR+\sum_{n_1+n_2= 1}\int_{\Sigma_T} \snabla^{n_1}(\partial_R^{\alpha})^{n_2}f\cdot  \snabla^{n_1}(\partial_R^{\alpha})^{n_2}\overline{g}\,Rd\sigma_{\mathbb{S}^2} dR,\\
			||f||^2_{\underline{H}^1_{\alpha}(\Sigma_T)}:=&\: \int_{\Sigma_T} R^{-2}|f|^2 \,Rd\sigma_{\mathbb{S}^2} dR+ \sum_{ n_1+n_2=1}\int_{\Sigma_T} |\snabla^{n_1}(\partial_R^{\alpha})^{n_2}f|^2\,Rd\sigma_{\mathbb{S}^2} dR,
		\end{align*}
		where the $\snabla_{\s^2}$ derivatives are to be interpreted as angular covariant derivatives on $\s^2$ in a weak sense, with respect to integration via the volume form $d\sigma_{\mathbb{S}^2}$ and  $\snabla:= \frac{1}{R} \snabla_{\s^2}$.
		
		\textbf{Note that the space $\underline{H}_{\alpha}^1(\Sigma_T)$ remains invariant as we change $\alpha$. That is to say
			\begin{equation*}
				\underline{H}_{\alpha}^1(\Sigma_T)=\underline{H}_0^1(\Sigma_T)=:\underline{H}^1(\Sigma_T)
			\end{equation*}
			for all $\alpha\in \R$.} Nevertheless, the definition of the $\underline{H}_{\alpha}^1(\Sigma_T)$ norm involving $\alpha$-twisted derivatives will be useful when deriving energy and elliptic estimates.
		
		Note moreover that $\underline{H}^1(\Sigma_T)$ is equal to the completion of $C_c^{\infty}(\Sigma_T)$ with respect to $||\cdot||_{\underline{H}^1(\Sigma_T)}$. We also denote the corresponding dual space with respect to the $\underline{H}^{1}(\Sigma_T)$ norm by $\underline{H}^{-1}(\Sigma_T):=(\underline{H}^{1}(\Sigma_T))^*$.

		We will also need to introduce weighted higher-order Sobolev spaces. We define
		\begin{align*}
			\underline{H}^2(\Sigma_T):=&\:\left\{R^{-1}f\in \underline{H}^1(\Sigma_T)\,\big|\, \snabla^{k_1}\partial_R^{k_2} f\in \underline{L}^2(\Sigma_T)\:\textnormal{for $ k_1+k_2= 2$}\right\},\\
			||f||_{\underline{H}^2(\Sigma_T)}^2:=&\: \sum_{k_1+k_2= 2}\int_{\Sigma_T} |\snabla^{k_1}\partial_R^{k_2}f|^2\,Rd\sigma_{\mathbb{S}^2} dR +||R^{-1}f||_{\underline{H}^1(\Sigma_T)}^2,\\
			\underline{H}^2_{\ell,p}(\Sigma_T):=&\:\left\{f\in \pi_{\geq \ell}(\underline{H}^1(\Sigma_T))\,\big|\, R^{-p}f_{\geq \ell+1}\in \underline{H}^2(\Sigma_T)\:\textnormal{and}\:R^{-p}\partial_R^{1-\alpha_{\ell}}\partial_R^{\alpha_{\ell}} f_{ \ell} \in  \underline{L}^2(\Sigma_T)\right\},\\
			||f||_{\underline{H}^2_{\ell,p}(\Sigma_T)}^2=&\: ||R^{-p}f_{\geq \ell+1}||_{\underline{H}^2(\Sigma_T)}^2+\int_{\Sigma_T} R^{-2p}|\partial_R^{1-\alpha_{\ell}}\partial_R^{\alpha_{\ell}}f_{ \ell}|^2\,R d\sigma_{\mathbb{S}^2} dR.
		\end{align*}
		The choice of $\underline{H}^2_{\ell,p}(\Sigma_T)$ is a little unusual, but it is natural from the point of view of elliptic estimates in Section~\ref{section.elliptic}. In the above expressions, we have implicitly extended the range of the projection operators $\pi_{\ell}$ to the space $\underline{L}^2(\Sigma_T)$.Note that  $R^{p}\ \pi_{\geq \ell}\left(\underline{H}^2(\Sigma_T)\right) \subset \underline{H}^2_{\ell,p}(\Sigma_T) $ with strict inclusion, however. Note also that for all $\alpha_{\ell}\in \R$, $\partial_R^{1-\alpha_{\ell}} \partial_R^{\alpha_{\ell}} f=\partial_R^2f+R^{-1}\partial_Rf-\alpha_{\ell}^2R^{-2}f$. Hence, we can characterize in the $p=0$ case:
		\begin{equation*}
			\underline{H}^2_{\ell,0}(\Sigma_T)= \underline{H}^1(\Sigma_T)\cap \{f_{\geq \ell+1}\in \underline{H}^2(\Sigma_T)\}\cap \{\partial_R^2f+R^{-1}\partial_Rf-\alpha_{\ell}^2R^{-2}f\in \underline{L}^2(\Sigma_T)\}.
		\end{equation*}

		We also define the relevant (weak) solution spaces, which are spacetime Banach spaces defined as follows: 
		\begin{align*} H^1_{\rm sol}:=&\ L^{\infty}( [-1,0]_T; \underline{H}^1(\Sigma_T))  \cap  W^{1,\infty}( [-1,0]_T; \LL)  \cap W^{2,\infty}( [-1,0]_T;\underline{H}^{-1}(\Sigma_T)) ,\\ 
			\| f \|_{H^1_{\rm sol}}:=&\ \textnormal{ess sup}_{T\in [-1,0]}\left[\| f(T,\cdot)\|_{\underline{H}^1(\Sigma_T)}+\| \partial_Tf(T,\cdot)\|_{\underline{L}^2(\Sigma_T)}+\| \partial_T^2f(T,\cdot)\|_{\underline{H}^{-1}(\Sigma_T))}\right].	
		\end{align*}
		
		We require moreover the following smaller solution spaces:
		\begin{align*} H^2_{\rm sol}:=&\ L^{\infty}( [-1,0]_T; \underline{H}^2_{0,0}(\Sigma_T))  \cap  W^{1,\infty}( [-1,0]_T; \underline{H}^1(\Sigma_T))  \cap W^{2,\infty}( [-1,0]_T;\LL) ,\\ 
			\| f \|_{H^2_{\rm sol}}:=&\ \textnormal{ess sup}_{T\in [-1,0]}\left[\| f(T,\cdot)\|_{\underline{H}^2_{0,0}(\Sigma_T)}+\| \partial_Tf(T,\cdot)\|_{\underline{H}^1(\Sigma_T)}+\| \partial_T^2f(T,\cdot)\|_{\LL}\right].	
		\end{align*}

		Note that the solution space $H^2_{\rm sol}$ is based on the space $\underline{H}^2_{0,0}(\Sigma_T)$, which is the regularity that is naturally preserved in evolution.		We list below the relevant initial data Hilbert spaces: 
		\begin{align*}
			H^1_{\rm data}:=&\ \{ (\Phi^0,\dot \Phi^0),\  \Phi^0  \in\underline{H}^1(\Sigma_{-1}), \dot \Phi^0 \in \LLi  \}, \\
			\| (\Phi^0,\dot \Phi^0) \|_{H^1_{\rm data}}^2:=&\ \| \Phi^0\|_{\underline{H}^1(\Sigma_{-1})}^2+  \| \dot \Phi^0\|_{\LLi}^2,\\
			H^2_{\rm data}:=&\ \{ (\Phi^0,\dot \Phi^0),\  \Phi^0  \in \underline{H}^2_{0,0}(\Sigma_{-1}), \dot \Phi^0 \in \underline{H}^1(\Sigma_{-1})  \}, \\
			\| (\Phi^0,\dot \Phi^0) \|_{H^2_{\rm data}}^2:=&\ \| \Phi^0\|_{\underline{H}^2_{0,0}(\Sigma_{-1})}^2+  \| \dot \Phi^0\|_{\underline{H}^1(\Sigma_{-1})}^2.
		\end{align*}

		We can inductively define initial data corresponding to higher-order $T$-derivatives as follows:
		\begin{equation}\label{T.deriv}
			\partial_T^{k}\Phi_{|T=-1}=\left[\partial_R^{1-\alpha_0}\partial_R^{\alpha_0}+R^{-2}\slashed{\Delta}_{\s^2}\right]\partial_T^{k-2}\Phi_{|T=-1}-2i qe R^{-1}\partial_T^{k-1}\Phi_{|T=-1},
		\end{equation}
		with $\ell\in \N \cup \{0\}$, $k\in \N$, $k\geq 2$ and $\partial_T^0\Phi|_{T=-1}=\Phi^0$ and $\partial_T^1\Phi|_{T=-1}=\dot \Phi^0$. After evolving suitably regular initial data $(\Phi_0, \dot \Phi_0)$, the expressions $\partial_T^{k}\Phi|_{T=-1}$ can be interpreted as restrictions of higher-order $T$-derivatives to the initial hypersurface.
		
		We define the following initial data space involving norms of higher-order $T$-derivatives:
		\begin{align*} &  	H^2_{{\rm data},0}= {H}^2_{data},\  	\|(\Phi^0,\dot \Phi^0)\|_{H^2_{{\rm data},0}}= 	\|(\Phi^0,\dot \Phi^0)\|_{H^2_{data}}, \\   
			&	H^2_{{\rm data},K}= H^2_{{\rm data},0 }
			\ \cap \left(\bigcap_{J=1}^K \{ (\Phi^0,\dot \Phi^0), \, \partial_T^{J}\Phi_{|T=-1} \in \underline{H}_{0,0}^2(\Sigma_{-1}),\,\partial_T^{1+J}\Phi_{|T=-1}\in \underline{H}^1(\Sigma_{-1}) \}\right) \text{ for any } K \in \mathbb{N}-\{0\},\\
			&	\|(\Phi^0,\dot \Phi^0)\|_{H^2_{{\rm data},K}}^2:= \|(\Phi^0,\dot \Phi^0)\|_{H^2_{data}}^2+\sum_{J=1}^K \|\partial_T^{J}\Phi_{|T=-1}\|_{\underline{H}_{0,0}^2(\Sigma_{-1})}^2+\|\partial_T^{1+J}\Phi_{|T=-1}\|_{ \underline{H}^1(\Sigma_{-1})}^2.
		\end{align*}
	\end{defn}
	
	Finally, when considering functions on the spheres $S_{T,R_0}:=\Sigma_T\cap\{R=R_0\}$, we will use the notation:
	\begin{equation*}
		||f||^2_{L^2(S_{T,R_0})}:=\int_{\s^2} |f|^2(T,R_0,\omega)\,d\sigma_{\s^2}.
	\end{equation*} 
	
	\subsection{Hardy inequality with respect to twisted derivatives}
	
	In this section, we consider arbitrary $\alpha \in \R$ and derive a Hardy-type inequality involving twisted derivatives, which will allow us to control weighted lower-order derivatives by higher-order $\alpha$-twisted derivatives.
	
	\begin{lemma}[$\alpha$-twisted Hardy inequality]
		\label{lm:Hardy}
		
		\begin{itemize}
			\item[\emph{(i)}] Let $\alpha\in \R$ and $p\in (-\infty,-\alpha)$, or let $\alpha\in [1,\infty)$ and $p\in (-\alpha,\infty)$. Then for all $f\in \underline{L}^2(\Sigma_T)$ with $R^{-p}\partial_{R}^{\alpha}f\in \underline{L}^2(\Sigma_T)$:
			\begin{equation}
				\label{eq:singL2control}
				\|R^{-1-p}f\|_{\underline{L}^2(\Sigma_T)}\leq \frac{1}{|\alpha+p|} \| R^{-p}\partial_R^{\alpha}f\|_{\underline{L}^2(\Sigma_T)}.
			\end{equation}
			In particular, we conclude that $R^{-p}f\in \underline{H}^1(\Sigma_T)$.
			\item[\emph{(ii)}]Let $\alpha\in \R$, $p\in \R\setminus\{-\alpha\}$ and assume that $R^{-p}f\in \underline{H}^1(\Sigma_T)$. Then \eqref{eq:singL2control} holds.
		\end{itemize}
	\end{lemma}
	\begin{proof}
		Suppose first that $f\in C_c^{\infty}(\Sigma_T)$. We integrate $R^{-2\alpha-2p}|R^{\alpha}f|^2$ on $(0,\infty)_R$ to obtain
		\begin{multline*}
			0=\int_{0}^{\infty} \partial_R( R^{-2\alpha-2p}|R^{\alpha}f|^2)\,dR\\
			=-(2\alpha+2p) \int_{0}^{\infty}R^{-1-2p-2\alpha}|R^{\alpha}f|^2\,dR+2 \int_{0}^{\infty}R^{-2\alpha-2p} \Re(\overline{R^{\alpha}f}\partial_R(R^{\alpha}f))\,dR.
		\end{multline*}
		We estimate
		\begin{equation*}
			2\left|\int_{0}^{\infty}R^{-2\alpha-2p}\Re(\overline{R^{\alpha}f}\partial_R(R^{\alpha}f))\,dR\right|\leq |\alpha+p|\int_{0}^{\infty} R^{-2p-1} |f|^2\,dR+\frac{1}{|\alpha+p|}\int_{0}^{\infty} R^{1-2p} |\partial_R^{\alpha}f|^2\,dR,
		\end{equation*}
		so we can rearrange terms to obtain for $\alpha+p\neq 0$:
		\begin{equation}
			\label{eq:hardcptsupp}
			\int_{0}^{\infty}R^{-1-2p}|f|^2\,dR\leq (p+\alpha)^{-2}\int_{0}^{\infty}R^{1-2p-2\alpha}|\partial_R(R^{\alpha}f)|^2\,dR,
		\end{equation}
		which gives \eqref{eq:singL2control} after integrating over $\mathbb{S}^2$.
		
		Note also that for $\alpha+p<0$ and for all $R_0>0$, we can similarly prove:
		\begin{equation}
			\label{eq:hardyalphpl0}
			\frac{1}{|\alpha+p|}|R_0^{-p}f|_{R=R_0}|^2+\int_{R_0}^{\infty}R^{-1-2p}|f|^2\,dR\leq (p+\alpha)^{-2}\int_{R_0}^{\infty}R^{1-2p-2\alpha}(\partial_R(R^{\alpha}f))^2\,dR.
		\end{equation}

		Now consider general $f\in \underline{L}^2(\Sigma_T)$ with $R^{-p}\partial_{R}^{\alpha}f\in \underline{L}^2(\Sigma_T)$. We will denote with $||\cdot||_{\underline{L}^2_{\geq R_0}(\Sigma_T)}$ the $\underline{L}^2(\Sigma_T)$ norm where the defining $R$-integrals are taken over the interval $[R_0,\infty)$, rather than $(0,\infty)$. 
		
		Then there exists a sequence $\{f_k\}$ in $C_c^{\infty}(\Sigma_T)$ such that, for fixed $R_0>0$, $\|R^{-p-1}(f-f_k)\|_{\underline{L}^2_{\geq R_0}(\Sigma_T)}+\|R^{-p}\partial_R^{\alpha}(f-f_k)\|_{\underline{L}^2_{\geq R_0}(\Sigma_T)}\to 0$ as $k\rightarrow +\infty$ and we apply \eqref{eq:hardyalphpl0} to $f_k$ to obtain:
		\begin{multline}
			\label{eq:densityest}
			||R^{-1-p}f||_{\underline{L}_{\geq R_0}^2(\Sigma_T)}\\
			\leq||R^{-1-p}f_k||_{\underline{L}_{\geq R_0}^2(\Sigma_T)}+ ||R^{-1-p}(f-f_k)||_{\underline{L}_{\geq R_0}^2(\Sigma_T)}\\
			\leq |\alpha+p|^{-1} \|R^{-p}\partial_R^{\alpha}f_k\|_{\underline{L}_{\geq R_0}^2(\Sigma_T)}+ \| R^{-p -1}(f-f_k)\|_{\underline{L}^2_{\geq R_0}(\Sigma_T)}\\
			\leq |\alpha+p|^{-1} \|R^{-p}\partial_R^{\alpha}f\|_{\underline{L}_{\geq R_0}^2(\Sigma_T)}+ \| R^{-p-1}(f-f_k)\|_{\underline{L}^2_{\geq R_0}(\Sigma_T)}+|\alpha+p|^{-1} \| R^{-p}\partial_R^{\alpha}(f-f_k)\|_{\underline{L}^2_{\geq R_0}(\Sigma_T)}.
		\end{multline}
		Given $\epsilon>0$, we can take $k\in \mathbb{N}$ suitably large so that for all $R_0>0$
		\begin{equation*}
			||R^{-1-p}f||_{\underline{L}_{\geq R_0}^2(\Sigma_T)}\leq |\alpha+p|^{-1} ||R^{-p}\partial_R^{\alpha}f||_{\underline{L}_{\geq R_0}^2(\Sigma_T)}+\epsilon.
		\end{equation*}
		Since this is true for all $\epsilon>0$, we conclude that \eqref{eq:singL2control} holds for all $R^{-p}f\in \underline{L}^2(\Sigma_T)$ provided $R^{-p}\partial_{R}^{\alpha}f\in \underline{L}^2(\Sigma_T)$, in the case $p+\alpha<0$.
		
		Consider now the case $p+\alpha>0$. Note that $f\in \underline{L}^2(\Sigma_T)$ implies that
		\begin{equation*}
			\int_{\Sigma_T} R^{1-2\alpha}|R^{\alpha}f|^2\,d\sigma_{\s^2}dR<\infty.
		\end{equation*}
		If $\alpha\geq1$, there must therefore exist a sequence $\{R_k\}$ with $R_k\downarrow 0$ as $k\to \infty$, such that
		\begin{equation*}
			|| R^{\alpha}f||_{L^2(S_{T,R_k})}\leq \frac{1}{k}.
		\end{equation*}
		By applying the fundamental theorem of calculus in the interval $[R_k,R]$, together with Cauchy--Schwarz, we therefore obtain:
		\begin{multline*}
			|| R^{\alpha}f||^2_{L^2(S_{T,R})}=\int_{\s^2}\left|  R^{\alpha}f(T,R,\omega)+\int_{R_k}^{R} \partial_R(R^{\alpha}f) (T,R',\omega)\,dR'\right|^2d\sigma_{\s^2}\\
			\leq 2 ||R^{\alpha}f||_{L^2(S_{T,R_k})}^2+2 \int_{R_k}^{R}R^{2(p+\alpha)-1}\,dR\int_{0}^R\int_{\s^2} R^{1-2p}|\partial_R^{\alpha}f|^2\,d\sigma_{\s^2}dR'\\
			\leq 2k^{-2}+ (p+\alpha)^{-1}R^{2(p+\alpha)}|| R^{-p}\partial_R^{\alpha}f||^2_{L^2_{\leq R}(\Sigma_T)},
		\end{multline*}
		where we used that $p+\alpha>0$. Since this is true for all $k$, we conclude that
		\begin{equation*}
			|| R^{-p}f||^2_{L^2(S_{T,R})}\leq 2(p+\alpha)^{-1}|| R^{-p-1}\partial_R^{\alpha}f||^2_{L^2_{\leq R}(\Sigma_T)}.
		\end{equation*}
		In particular, this implies that $\lim_{R\downarrow 0}|| R^{-p}f||^2_{L^2(S_{T,R})}=0$. We can therefore apply the derivation of \eqref{eq:hardcptsupp} to conclude that \eqref{eq:singL2control} holds also for $f\in \underline{L}^2(\Sigma_T)$ if $R^{-p}\partial_{R}^{\alpha}f\in \underline{L}^2(\Sigma_T)$, with $p+\alpha>0$ and $\alpha\geq1$.
		
		We turn to (ii). Since $R^{-p}f\in \underline{H}^1(\Sigma_T)$, there exists a sequence $\{f_k\}$ in $C^{\infty}_c(\Sigma_T)$ such that $||R^{-p-1}(f-f_k)||_{\underline{L}^2(\Sigma_T)}\to 0$ and $||R^{-p}\partial_R^{\alpha}(f-f_k)||_{\underline{L}^2(\Sigma_T)}\to 0$. We can therefore repeat the derivation of \eqref{eq:densityest}, but without the restriction $\geq R_0$, appealing to \eqref{eq:hardcptsupp} instead of \eqref{eq:hardyalphpl0}.
		
	\end{proof}

	\section{Main results}
	In this section, we prove the main results of the paper. In Section \ref{section.elliptic}, we establish elliptic estimates necessary to derive higher regularity for solutions to \eqref{CSF.4}. Then in Section \ref{sec:wellposedness}, we establish the wellposedness theory for \eqref{CSF.4}, first in low-regularity spaces and then using Section \ref{section.elliptic}, also in higher-regularity spaces.

	\subsection{Reduction to a problem on $AdS_2\times \mathbb{S}^2$}
	In this section, we relate the initial data for $(\phi^0,\dot \phi^0)$, as prescribed in Theorem~\ref{main.theorem}, to initial data $(\Phi^0,\dot \Phi^0)$ for \eqref{CSF.4} contained in the function spaces introduced in Section \ref{sec:functionspaces}. This will allow us to translate properties of solutions of \eqref{CSF.4} to solutions of \eqref{CSF}, arising from appropriate initial data.

	\begin{lemma} \label{lemma.conversion}
		Let $(\phi^0, \dot \phi^0) \in \mathcal{H}^1_{data}$. We can interpret $(\phi^0, \dot  \phi^0)$ as functions on $\{T=-1\}$ that are trivial for $R>1$. Let
		\begin{equation*}
			(\Phi^0,\dot \Phi^0):=(\phi^0, \dot \phi^0)*\chi_{\delta},
		\end{equation*}
		with $0<\delta<\frac{1}{2}$ and $\chi_{\delta}:[0,2]_R\to \R$ a mollifer that vanishes for $R\notin [1-\delta,1+\delta]$. Then:
		\begin{itemize}
			\item[\emph{(i)}]There exists a numerical constant $C>0$ (independent of $\delta$), such that
			\begin{equation*}
				||(\Phi^0,\dot \Phi^0)||_{\underline{H}^1_{\rm data}}\leq C||(\phi^0, \dot \phi^0)||_{\mathcal{H}^1_{\rm data}}.
			\end{equation*}

			\item[\emph{(ii)}] 	 If we assume additionally that $(\phi^0, \dot \phi^0) \in \mathcal{H}^2_{data}$, then there exists a numerical constant $C>0$ (independent of $\delta$), such that			\begin{equation*}				\begin{split}					||(\Phi^0,\dot \Phi^0)||_{\underline{H}^2_{\rm data}}\leq &C \|(\phi^0, \dot \phi^0)\|_{\mathcal{H}^2_{\rm data}	}.							\end{split}			\end{equation*}
			\item[\emph{(iii)}]If we assume additionally that for $k\in \N\cup\{0\}$: $(\phi^0, \dot \phi^0) \in \mathcal{H}^2_{data,k}$, then there exists a numerical constant $C>0$ (independent of $\delta$), such that
			\begin{equation*}
				||(\Phi^0,\dot \Phi^0)||_{\underline{H}^2_{\rm data,k}}\leq C||(\phi^0, \dot \phi^0)||_{\mathcal{H}^{2}_{\rm data,k}}.
			\end{equation*}
		\end{itemize}
	\end{lemma}
	\begin{proof}
		Let $f\in \underline{L}^2(\{T=-1\})$ such that $f|_{R>1}\equiv 0$. Then we can estimate
		\begin{equation*}
			\int_{\{T=-1\}} |f*\chi_{\delta}|^2\,d\omega dR\leq \int_{\Sigma_T}\frac{|f|^2}{(1+r)^2}\,d\omega dr.
		\end{equation*}
		Furthermore, there exist numerical constants $C,c>0$, such that
		\begin{equation*}
			c\frac{r}{1+r}\leq R\leq C\frac{r}{1+r}.
		\end{equation*}
		Using the definition of $\Phi$, it is therefore straightforward to show for example that:
		\begin{equation*}
			\int_{\{T=-1\}} |\Phi^0|^2R^{-1}\,d\omega dR\leq \int_{\Sigma_T}\frac{|f|^2}{r^2}\,d\omega dr.
		\end{equation*}
		Using additionally that $\partial_R=K^{\perp}$, we conclude (i). We obtain (ii) and (iii) restricted to spherical harmonics of degree $\geq 1$ in a similar manner, after considering additional $\partial_R$ derivatives. To estimate the spherical mean in (ii), we additionally apply \eqref{CSF.4} to express:
		\begin{equation*}
			\partial_R^{1-\alpha_0} \partial_R^{\alpha_0}  \pi_0\Phi=K^2   \pi_0\Phi+2i q eR^{-1} K \pi_0\Phi.
		\end{equation*}
	\end{proof}

	\subsection{Weak solutions of a $\alpha$-twisted Laplace problem with Dirichlet boundary conditions and elliptic estimates} \label{section.elliptic}
	
	We consider the following elliptic equation:
	\begin{align} 
		\label{eq:Lalphainpf}
		\mathcal{L} =&\:\partial_R^{1-\alpha_{0}}\partial_R^{\alpha_{0}}+R^{-2}\slashed{\Delta}_{\s^2},\\ 
		\mathcal{L} \Phi=&\:f \label{def.Dirichlet},
	\end{align}
	where the inhomogeneity $f$ lies in a function space that will be specified below.
	\begin{rmk} \label{rmk.Dirichlet/wave}
		Note that, classical solutions $\Phi$ of \eqref{CSF.4} satisfy \eqref{def.Dirichlet} with $f= \partial_T^2 \Phi + \frac{2iqe\ \partial_T \Phi}{R} $.
	\end{rmk}
	We now introduce the following bilinear form for all $T\in [-1,0]$. \begin{align*}	B_T:&\:\underline{H}^1(\Sigma_T)\times \underline{H}^1(\Sigma_T)\to \R,\\
		B_T[w_1,w_2]:=&\:\int_{\Sigma_T}\left[\partial_{R}^{\alpha_{0}} w_1\ \partial_R^{\alpha_{0}}\overline{w_2}+R^{-2}\snabla_{\s^2}w_1 \snabla_{\mathbb{S}^2}\overline{w_2}\right]\,Rd\sigma_{\mathbb{S}^2} dR.
	\end{align*}
	\begin{defn} \label{def.weak}
		Let $T \in [-1,0]$, and $f \in  \underline{H}^{-1}(\Sigma_T)$.  We say $\Phi \in  \underline{H}^{1}(\Sigma_T) $ is a weak solution of \eqref{def.Dirichlet} on $\Sigma_T$ satisfying Dirichlet boundary conditions, if for all $w\in \underline{H}^{1}(\Sigma_T)$: \begin{equation} \label{weak.eq}
			B_T[\Phi,w]=\:-f(w),
		\end{equation}
		where we applied the notation introduced in Definition~\ref{def.Sobolev}
	\end{defn}
	
	\begin{rmk}\label{rmk.weak.l}
		Observe that the projection operators $\pi_{\ell}$ commute with the operator $\mathcal{L}$ for any $\ell \in  \N_0$, so $\Phi_{\ell}$ and $\Phi_{\geq \ell}$ are also weak solutions of \eqref{def.Dirichlet} if $\Phi$ is a weak solution of \eqref{def.Dirichlet}. We will make frequent use of this fact below.
	\end{rmk}

	\begin{prop}\label{lemma.energy.ang}
		Let $\alpha_0\neq 0$ and $f \in\HHd$. There exists a unique weak solution $\Phi \in  \HH$ of \eqref{def.Dirichlet}, and \begin{equation} \label{energ1}
			\| \Phi\|_{\HH} \leq {C} \| f \|_{\HHzd}
		\end{equation} for some constant ${C}(\alpha_{0})>0$. 	If moreover  $f \in \LL$ then \begin{equation} \label{energ2}
			\| \Phi\|_{\HH} \leq C \| f \|_{\LL}
		\end{equation} 
	\end{prop}
	
	\begin{proof}
		We will apply the Lax--Milgram theory on the Hilbert space $\HH$.	The only non-trivial element to check is the coercivity of $B_T$ on the Hilbert space $\HH$. 
		
		First recall that (ii) of Lemma~\ref{lm:Hardy} implies the following estimate for any $h\in \underline{H}^1(\Sigma_T)$: $$\frac{1}{2}\| \partial_R^{\alpha_0} h \|_{\LL}^2 \geq \frac{1}{2}\alpha_{0}^2  \|{R^{-1}}h \|_{\LL}^2,$$ from which it follows that
		\begin{align*}
			B_T(h,h) \geq & \frac{1}{2}	\| \partial_R^{\alpha_0} h \|_{\LL}^2+\frac{\alpha^2_0}{2}  \|R^{-1}h \|_{\LL}^2+ \|R^{-1}\snabla_{\mathbb{S}^2} h\|_{\LL}^2 \geq \frac{\min\{1,\alpha_{0}^2\}}{2}  \| h\|_{\underline{H}^1_{\alpha_0}}^2, 
		\end{align*}  so the coercivity is established after using that the norms $ \| \cdot \|_{\underline{H}^1_{\alpha_0}}$ and $\| \cdot \|_{\HH}$ are equivalent. This gives $\Phi \in \HHz$	 and \eqref{energ1}, \eqref{energ2} follow.
	\end{proof}

	In the proposition below, we obtain higher regularity for solutions of \eqref{def.Dirichlet} via elliptic estimates.
	\begin{prop}\label{lemma.elliptic.ang}
		Let $\ell\in \N\cup\{0\}$ and $p+1\in (-\infty, \alpha_{\ell+1})\setminus\{-\alpha_{\ell}\}$. Assume that $f \in \LL$ and $R^{-p}f_{\geq \ell} \in  \pi_{\geq \ell}(\LL)$ and let $\Phi \in   \underline{H}^{1}(\Sigma_T) $  be the unique weak solution  of \eqref{def.Dirichlet}. Then: 
		\begin{enumerate}[\rm (i)]
			\item 
			We have that $\Phi_{\geq \ell}\in \underline{H}^2_{\ell,p}(\Sigma_T)$
			and there exists $C=C(\ell,p)>0$  such that 
			\begin{equation}
				\label{est.elliptic.ang}
				\|  \Phi_{\geq \ell} \|_{ \underline{H}^2_{{\ell},p}(\Sigma_T)} \leq C||R^{-p}f_{\geq \ell}||_{\underline{L}^2(\Sigma_{T})} .
			\end{equation}
			\item 
			
			If moreover $p+1<\alpha_{\ell}$ and in the case $\ell=0$, assume additionally that $\alpha_0\geq1$ if $p+1>-\alpha_0$, then there exists $C=C(\ell,p)>0$  such that
			\begin{equation}
				\label{est.elliptic.ang2}
				\| R^{-p} \Phi_{\geq \ell} \|_{ \underline{H}^2(\Sigma_T)} \leq C||R^{-p}f_{\geq \ell}||_{\underline{L}^2(\Sigma_{T})} .
			\end{equation}
			
		\end{enumerate}
	\end{prop}
	
	\begin{proof}Without loss of generality, we can derive uniform estimates below by considering $\mathcal{L} w$ for functions $w\in ~\pi_{\geq \ell}(C^{\infty}_c(\Sigma_T))$, and then conclude that $\Phi_{\geq \ell}\in \underline{H}^2_{\ell,p}(\Sigma_T)$ by a standard derivative-quotient argument. 
		
		Using the definition of $\mathcal{L}$ from \eqref{eq:Lalphainpf} together with the identity:
		\begin{equation*}
			\partial_R^{1-\alpha_{0}}\partial_R^{\alpha_{0}}=\partial_R^{1-\alpha_{\ell}}\partial_R^{\alpha_{\ell}}+R^{-2}\ell(\ell+1),
		\end{equation*}
		we can write:
		\begin{equation}
			\label{eq:Lell}
			\mathcal{L}w=\partial_R^{1-\alpha_{\ell}}\partial_R^{\alpha_{\ell}}w+R^{-2}(\slashed{\Delta}_{\s^2}+\ell(\ell+1))w.
		\end{equation}
		
		By taking the square norm on both sides and multiplying by the factor $R^{1-2p}$, with $p\geq 0$, we obtain
		\begin{equation*}
			R^{1-2p}|\mathcal{L}w|^2=R^{1-2p}|\partial_R^{1-\alpha_{\ell}}\partial_R^{\alpha_{\ell}}w|^2+R^{-3-2p}|(\slashed{\Delta}_{\s^2}+\ell(\ell+1))w|^2+2R^{-1-2p}\Re(\partial_R^{1-\alpha_{\ell}}\partial_R^{\alpha_{\ell}}w(\slashed{\Delta}_{\s^2}+\ell(\ell+1))\overline{w}).
		\end{equation*}
		We integrate the above equation. By integrating by parts, using that the boundary terms vanish by the compactness of the support of $w$, we obtain:
		\begin{equation*}
			\begin{split}
				\int_{\Sigma_T}&2R^{-1-2p}\Re(\partial_R^{1-\alpha_{\ell}}\partial_R^{\alpha_{\ell}}w (\slashed{\Delta}_{\s^2}+\ell(\ell+1))\overline{w}) \, d\sigma_{\s^2} dR\\
				=&\:\int_{\Sigma_T}2\Re\left[R^{-2-2p}\partial_R(R^{1-2\alpha_{\ell}}\partial_R(R^{\alpha_{\ell}}w))(\slashed{\Delta}_{\s^2}+\ell(\ell+1))\overline{R^{\alpha_{\ell}}w}\right] \,d\sigma_{\s^2} dR\\
				=&\:2 \int_{\Sigma_T}R^{-1-2\alpha_{\ell}-2p}(|\snabla_{\s^2}\partial_R(R^{\alpha_{\ell}}w)|^2-\ell(\ell+1)|\partial_R(R^{\alpha_{\ell}}w)|^2)\,d\sigma_{\s^2} dR\\
				&-2(1+p)\int_{\Sigma_T}R^{-2-2\alpha_{\ell}-2p}\partial_R(|\snabla_{\s^2}R^{\alpha_{\ell}}w|^2-\ell(\ell+1)|R^{\alpha_{\ell}}w|^2)\,d\sigma_{\s^2} dR\\
				=&\: \int_{\Sigma_T}\bigl[2R^{-1-2\alpha_{\ell}-2p}(|\snabla_{\s^2}\partial_R(R^{\alpha_{\ell}}w)|^2-\ell(\ell+1)|\partial_R(R^{\alpha_{\ell}}w)|^2)\\
				&-4(1+p)(1+\alpha_{\ell}+p)R^{-3-2p}(|\snabla_{\s^2}w|^2-\ell(\ell+1)|w|^2) \bigr]d\sigma_{\s^2} dR.
			\end{split}
		\end{equation*}
		We therefore have the following identity:
		\begin{equation}
			\begin{split}
				\int_{\Sigma_T}& |\mathcal{L}w|^2\,R^{1-2p} d\sigma_{\s^2} dR\\
				=&\:\int_{\Sigma_T} \underbrace{\Biggl[|\partial_R^{1-\alpha_{\ell}}\partial_R^{\alpha_{\ell}}w|^2}_{I}+\underbrace{2R^{-2}(|\snabla_{\s^2}\partial_R^{\alpha_{\ell}}w|^2-\ell(\ell+1)|\partial_R^{\alpha_{\ell}}w|^2)}_{II}+\underbrace{R^{-4}|(\slashed{\Delta}_{\s^2}+\ell(\ell+1))w|^2}_{III}\\
				&\underbrace{-4(1+p)(1+\alpha_{\ell}+p)R^{-4}(|\snabla_{\s^2}w|^2-\ell(\ell+1)|w|^2)\Biggr]}_{IV}\,R^{1-2p} d\sigma_{\s^2} dR.
			\end{split}
		\end{equation}
		Note that $I$ and $III$ are non-negative definite, while $\int_{\Sigma_T} II  \,R^{1-2p} d\sigma_{\s^2} dR >0$ by Lemma~\ref{lem.poincare}. The only  term  that requires an estimate is  $\int_{\Sigma_T} IV  \,R d\sigma_{\s^2} dR $. We will show that we can absorb this term into  $\int_{\Sigma_T} II  \,R^{1-2p} d\sigma_{\s^2} dR $ and $\int_{\Sigma_T} III  \,R^{-3-2p} d\sigma_{\s^2} dR $.

		We first decompose $w= \sum_{\ell'=\ell}^{+\infty} w_{\ell'}$. By \eqref{poinc2}, we obtain the identity
		
		\begin{equation*}\begin{split}
				\int_{\Sigma_T} [|\slashed{\nabla}_{\mathbb{S}^2} w_{\ell'}|^2-\ell(\ell+1)|w_{\ell'}|^2]  R^{-3-2p} d\sigma_{\s^2} dR =\ & [\ell'(\ell'+1)-\ell(\ell+1)] 	\int_{\Sigma_T} | w_{\ell'}|^2   R^{-3-2p} d\sigma_{\s^2} dR.\end{split}
		\end{equation*}
		
		Applying (ii) of Lemma~\ref{lm:Hardy} to the RHS gives the following inequality, using that $p+1\neq -\alpha_{\ell}$:
		\begin{equation*}\begin{split}
				\int_{\Sigma_T} |  w_{\ell'}|^2   R^{-3-2p} d\sigma_{\s^2} dR \leq\ & \frac{1}{ (1+\alpha_{\ell}+p)^{2}} 	\int_{\Sigma_T} | \partial_R^{\alpha_{\ell}}  w_{\ell'}|^2   R^{-1-2p} d\sigma_{\s^2} dR.\end{split}
		\end{equation*}
		
		Now applying \eqref{poinc2} again, this time on both sides of the inequality finally gives:
		\begin{multline*}
			\int_{\Sigma_T} [|\slashed{\nabla}_{\mathbb{S}^2} w_{\ell'}|^2-\ell(\ell+1)|w_{\ell'}|^2]  R^{-3-2p} d\sigma_{\s^2} dR\\
			\leq\  \frac{1}{ (1+\alpha_{\ell}+p)^{2}} 	\int_{\Sigma_T} [| \partial_R^{\alpha_{\ell}}\slashed{\nabla}_{\mathbb{S}^2} w_{\ell'}|^2-\ell(\ell+1)|\partial_R^{\alpha_{\ell}}w_{\ell'}|^2]  R^{-1-2p} d\sigma_{\s^2} dR.
		\end{multline*}
		
		Since the constant in the above estimate does not depend on $\ell'$,  we can sum over $\ell'$ to obtain  
		\begin{equation}
			\label{eq:hardyelliptic}
			\int_{\Sigma_T} [|\slashed{\nabla}_{\mathbb{S}^2} w|^2-\ell(\ell+1)|w|^2]  R^{-3-2p} d\sigma_{\s^2} dR \leq\   \frac{1}{ (1+\alpha_{\ell}+p)^{2}} 	\int_{\Sigma_T} [| \partial_R^{\alpha_{\ell}}\slashed{\nabla}_{\mathbb{S}^2} w|^2-\ell(\ell+1)|\partial_R^{\alpha_{\ell}}w|^2]  R^{-1-2p} d\sigma_{\s^2} dR.
		\end{equation}
		
		We can relate $\partial_R^{1-\alpha_{\ell}}\partial_R^{\alpha_{\ell}}w$ and $(\slashed{\Delta}_{\s^2}+\ell(\ell+1))w$ by applying \eqref{eq:Lell}:
		\begin{equation*}
			\begin{split}
				R^{-4}|(\slashed{\Delta}_{\s^2}+\ell(\ell+1))w|^2=&\:|\mathcal{L}w-\partial_R^{1-\alpha_{\ell}}\partial_R^{\alpha_{\ell}}w|^2\\
				\leq &\:(1+\delta)|\partial_R^{1-\alpha_{\ell}}\partial_R^{\alpha_{\ell}}w|^2+(1+\delta^{-1})|\mathcal{L}w|^2,
			\end{split}
		\end{equation*}
		and then multiplying both sides by $\frac{1-\delta}{1+\delta}$ in order to combine the terms I and III and obtain control over the integral of
		\begin{equation}
			\label{eq:absorb2ndder}
			\delta R^{1-2p}|\partial_R^{1-\alpha_{\ell}}\partial_R^{\alpha_{\ell}}w|^2+2\left(1-\frac{\delta}{1+\delta}\right)R^{-3-2p}|(\slashed{\Delta}_{\s^2}+\ell(\ell+1))w|^2-C_{\delta}R^{1-2p}|\mathcal{L}w|^2,
		\end{equation}
		for an appropriately large constant $C_{\delta}>0$.
		
		Noting that $\slashed{\Delta}_{\s^2} w_{ \ell}+\ell(\ell+1))w_{ \ell} =0$, we now apply \eqref{poinc2} to $w_{\geq \ell+1}$ to control:
		\begin{equation}
			\label{eq:poincarelliptic}
			\begin{split}
				\int_{\s^2}&2\left(1-\frac{\delta}{1+\delta}\right)R^{-3-2p}|(\slashed{\Delta}_{\s^2}+\ell(\ell+1))w_{\geq \ell+1}|^2\,d\sigma_{\s^2}\\
				\geq &\: 4(\ell+1)
				\left(1-\frac{\delta}{1+\delta}\right)\int_{\s^2}[|\slashed{\nabla}_{\mathbb{S}^2} w_{\geq \ell+1}|^2-\ell(\ell+1)|w_{\geq \ell+1}|^2]  R^{-3-2p} \,d\sigma_{\s^2}.
			\end{split}
		\end{equation}
		
		Hence, combining the estimate \eqref{eq:hardyelliptic}, \eqref{eq:absorb2ndder} and \eqref{eq:poincarelliptic}, we obtain the estimate
		\begin{multline*}
			\int_{\Sigma_T} \left[2 (1+\alpha_{\ell}+p)^{2} + 4(\ell+1)\left(1-\frac{\delta}{1+\delta}\right)-4(1+p)(1+\alpha_{\ell}+p)\right] R^{-3+2p}\\
			\times (|\snabla_{\s^2}w_{\geq \ell+1}|^2-\ell(\ell+1)|w_{\geq \ell+1}|^2)\,d\sigma_{\s^2} dR\\
			\leq C \int_{\Sigma_T} R^{1-2p}|\mathcal{L}w_{\geq 
				\ell+1}|^2\,d\sigma_{\s^2}dR.
		\end{multline*}
		In order to be able to take $\delta>0$ suitably small so as to guarantee non-negativity of the integrand on the LHS, we require that:
		\begin{equation*}
			2 (1+\alpha_{\ell}+p)^{2} + 4(\ell+1)-4(1+p)(1+\alpha_{\ell}+p)>0.
		\end{equation*}
		It is straightforward to show that this is equivalent to the condition:
		\begin{equation*}
			p+1<\sqrt{\alpha_{\ell}^2+2(\ell+1)}=\alpha_{\ell+1},
		\end{equation*}
		which completes the proof of (i) for $w_{\geq 
			\ell+1}$. It remains to consider $w_{\ell}$.
		
		By \eqref{poinc2}, we have the following identity: for any $q\in \R$
		\begin{equation*}
			\begin{split}
				\int_{\Sigma_T}& |\mathcal{L}w_{\ell}|^2\,R^{1-2q} d\sigma_{\s^2} dR=\int_{\Sigma_T} |\partial_R^{1-\alpha_{\ell}}\partial_R^{\alpha_{\ell}}w_{\ell}|^2 R^{1-2q} d\sigma_{\s^2} dR,
			\end{split}
		\end{equation*}
		so, in particular, (i) follows for $w_{\ell}$ if we take $q=p$.
		
		Now we turn to (ii). By what precedes, we already have $$ \| R^{-p} \Phi_{\geq \ell+1}\|_{\HHH} \leq C \| R^{-p} \mathcal{L}\Phi \|_{\LL},$$ for $p+1<\alpha_{\ell+1}$. This moreover implies \eqref{est.elliptic.ang2} for $p+1<\alpha_{\ell}$ in the case $\ell\geq1$.

		Hence, thus we only need to consider $\Phi_{ \ell}$ with $\ell=0$. Applying (i) of Lemma~\ref{lm:Hardy}, using that $\partial_R^{\alpha_{0}}\Phi_{0}\in \underline{L}^2(\Sigma_T)$ and $R^{-p}\partial_R^{1-\alpha_{0}}\partial_R^{\alpha_{0}}\Phi_{0}\in \underline{L}^2(\Sigma_T)$ and $p+1<\alpha_{0}$, we obtain
		\begin{equation}\label{hardy.problem0}		
			\int_{\Sigma_T} |\partial_R^{1-\alpha_{0}}\partial_R^{\alpha_{0}}\Phi_{0}|^2 R^{1-2p} d\sigma_{\s^2} dR \geq  (\alpha_0-1-p)^2  \int_{\Sigma_T} |\partial_R^{\alpha_{0}}\Phi_{\ell}|^2 R^{-1-2p} d\sigma_{\s^2} dR.
		\end{equation} 
		So, in particular, $||R^{-p-1}\partial_R^{\alpha_{0}}\Phi_{0}||_{\underline{L}^2(\Sigma_{T})}\leq C||R^{-p}f||_{\underline{L}^2(\Sigma_{T})}$. Since $\Phi_0\in \underline{L}^2(\Sigma_T)$,  $\alpha_0\geq1$ and $p+1>-\alpha_0$, or $p+1<-\alpha_{0}$ with no restriction on $\alpha_0$, we can apply (i) of Lemma~\ref{lm:Hardy} to obtain also $||R^{-p-2}\Phi_0||_{\underline{L}^2(\Sigma_T)}\leq C ||R^{-p}f||_{\underline{L}^2(\Sigma_{T})}$.

	\end{proof}

	\subsection{Existence and uniqueness for \eqref{CSF.4} with Dirichlet boundary conditions}
	\label{sec:wellposedness}
	
	The main purpose of this section is the local wellposedness of the wave equation on $\widetilde{\mathcal{M}}\subset AdS_2\times \s^2$ in the Sobolev space $H^2_{sol}$ (Dirichlet boundary conditions), with additional angular regularity, i.e.\ the proof of the estimate \eqref{main.H2.bound}. 
	
	\begin{defn}
		\label{def;weaksol}
		Suppose $(\Phi_0,\dPhiz) \in\underline{H}^1(\Sigma_{-1}) \times   \LLi$, and $\Phi \in  H^1_{sol}$. We say that $\Phi$ is a weak solution of \eqref{CSF.4} arising from initial data $(\Phi_0,\dPhiz)$ with Dirichlet boundary conditions if \begin{enumerate}[i)]
			\item for all $w \in \HHz$ and for almost every $T \in [-1,0]$:
			\begin{equation} \label{eq.weak.wave}
				\left\la	\partial_T^2 \Phi,w\right\ra_{\underline{H}^{-1}(\Sigma_T)/\underline{H}^{1}(\Sigma_T)}+\left\la \frac{2i qe}{R} \partial_T \Phi, w\right\ra_{\underline{L}^2(\Sigma_T)}+ B_T[\Phi,w]=0,
			\end{equation} where  $\left\la \cdot, \cdot\right\ra_{\underline{H}^{-1}(\Sigma_T)/\underline{H}^{1}(\Sigma_T)}$ denotes the duality bracket between $\underline{H}^{1}(\Sigma_T)$ and its dual $\underline{H}^{-1}(\Sigma_T)$.
			\item $(\Phi,\partial_T \Phi)_{|\Sigma_{-1}} = (\Phi_0,\dPhiz)$.
		\end{enumerate}

	\end{defn}
	
	Note that \eqref{eq.weak.wave} makes sense for $\Phi\in H^1_{\rm sol}$, since we can write
	\begin{equation*}
		\left\la \frac{2i qe}{R} \partial_T \Phi, w\right\ra_{\underline{L}^2(\Sigma_T)}=\la 2i qe \partial_T \Phi, R^{-1}w\ra_{\underline{L}^2(\Sigma_T)},
	\end{equation*}
	and we have that $\partial_T \Phi \in\LL$ and, since  $w\in   \underline{H}^1(\Sigma_{T})$, we also have that $R^{-1}w\in\LL$. Furthermore, the condition ii) is also well-defined, since it follows from standard arguments (see for example \cite{Evans}[Theorem 2, \S 5.9.2]) that $\Phi\in H^1_{\rm sol}$ implies that $\Phi\in C^0([-1,0];L^2(\Sigma_T))$ and $\partial_T\Phi\in ~C^0([-1,0]; H^{-1}(\Sigma_T))$.

	In the proposition below, we establish existence and uniqueness of weak solutions by applying arguments developed in \cite{WarnickAdS} in the setting of well-posedness of wave equations on asymptotically anti-de Sitter spacetimes.
	
	\begin{prop}[Existence/uniqueness of weak solutions for the wave equation on $AdS_2\times \mathbb{S}^2$]
		\label{prop.existence.wave}
		Let $(\Phi_0,\dPhiz) \in H^1_{data}$. Then there exists a unique weak solution $\Phi \in H^1_{sol}$ of \eqref{CSF.4} with data $(\Phi_0,\dPhiz)$. Moreover, there exists a constant $C>0$, independent of $\Phi$, such that  \begin{equation} \label{energy}
			\| \Phi \|_{H^1_{sol}} \leq C   \| (\Phi_0, \dPhiz)\|_{H^1_{data}} .
		\end{equation}
	\end{prop}
	\begin{proof}
		The uniqueness proof is identical to that of \cite{WarnickAdS}[Theorem 4.1], which itself follows from standard arguments for uniqueness of linear wave equations (see Evans \cite{Evans}[Theorem 4, \S 7.2]): the only difference is that, since we work with complex-valued functions, we need to consider the complex conjugate when applying vector field multipliers and correspondingly, we need take the real part of the resulting integrals. In the proof of \cite{Evans}, one will see that the only new contribution coming from the $iqe\ R^{-1} \partial_T \Phi$ term vanishes in the process, so the proof can be repeated. 
		
		The proof of existence is also identical to \cite{WarnickAdS}[Theorem 4.2] (see also \cite{Evans}[Theorems 3.5, \S 7.2]), where solutions are constructed as appropriate limits of a sequence of smooth solutions $\Phi_k$ to \eqref{CSF.4} on the subset $\{R\geq R_k>0\}$, where $R_k\downarrow 0$ as $k\to \infty$ and we assume $\Phi_k|_{R=R_k}\equiv 0$. The reason why the results of \cite{WarnickAdS} apply also to \eqref{CSF.4}  is that, after taking the real part, the energy identity corresponding to \eqref{CSF.4} is the same as the energy identity in the equations considered in  \cite{WarnickAdS}; the additional term  $iqe\ R^{-1} \partial_T \Phi$ does not contribute. 
		
		The only other difference with the proof of Theorem 4.2 in \cite{WarnickAdS} involves the existence of weak limits of $R^{-1}\partial_T\Phi_k$. Indeed, if we take a sequence $\Phi_k  \in H^1_{sol}$, such that  $\sup_{k\in \mathbb{N}}\| \Phi_k \|_{H^1_{sol}} <\infty$, we need to show additionally that for all $w \in \HHz$,  $$ \langle  R^{-1}\partial_T\Phi_k, w\rangle $$ has a limit (using a compactness argument as in Lemma 4.4.3 in \cite{WarnickAdS}). Since $\Phi_k \in H^1_{sol}$, $\| \partial_T \Phi_k \|_{\underline{L}^2(\Sigma_T)}$  is uniformly bounded in $k$, hence $\| R^{-1}\partial_T\Phi_k \|_{\HHd}$ is uniformly bounded in $k$, hence we can argue again by weak compactness of the unit ball in $\HHd$.
	\end{proof}

	\begin{prop}[Higher regularity for the wave equation] \label{prop.H2}
		Let $\Phi \in H^1_{sol} $ the unique solution of \eqref{CSF.4} with data $(\Phi_0,\dPhiz) \in H^2_{data,K}$ for some $K\in \mathbb{N}_0$. Then, for all $0 \leq J\leq K$, $\partial_T^J \Phi_{\geq \ell} \in H^2_{sol} $ and there exists $C(\alpha_0,K)>0$ such that, in the notation $\partial_T^J\Phi_0$, $\partial_T^J\dot{\Phi}_0$ of Definition~\ref{def.Sobolev}:
		\begin{equation}\label{H2.est.noang}
			\sum_{J=0}^K\|\partial_T^J \Phi \|_{H^2_{sol}} \leq C  \| (\Phi_0,  \dPhiz)\|_{H^2_{data,K}}.
		\end{equation}
		
		If moreover  $(\snabla_{\mathbb{S}^2}^n\Phi_0,\snabla_{\mathbb{S}^2}^n\dPhiz) \in H^2_{data,K}$ for $0\leq n \leq 2$, then $\partial_T^J\snabla_{\mathbb{S}^2}^n\Phi \in H^2_{sol} $ for $0\leq n \leq 2$ and $0 \leq J \leq K$, and
		\begin{equation} \label{main.H2.bound}
			\sum_{J=0}^K	\sum_{k=0}^{2}\| \snabla_{\mathbb{S}^2}^k\Phi\|_{H^2_{sol}} \leq C   \sum_{k=0}^{2}\| (\snabla_{\mathbb{S}^2}^k \Phi_0,\snabla_{\mathbb{S}^2}^k,\dPhiz)\|_{H^2_{data,K}}.
		\end{equation} 
	\end{prop}

	\begin{proof}
		
		The proof of \eqref{H2.est.noang} proceeds exactly as in \cite{WarnickAdS}[Theorem 5.5], after commutation with  $\partial_T$ and use of the elliptic estimates recalled in Section~\ref{section.elliptic}. Note that, as in \cite{WarnickAdS}, $\partial_T$ leaves \eqref{CSF.4} invariant.
		The last estimate \eqref{main.H2.bound} then follows from the commutation with the two angular Killing vector fields on the sphere $\mathbb{S}^2$.
		
	\end{proof}

	\subsection{Pointwise boundedness and continuity for $\alpha_0<1$}

	In this section, we derive quantitative estimates for the behavior near $R=0$ of solutions $\Phi$ to \eqref{CSF.4} or \eqref{wave.pot3} with $\alpha_0<1$. The case $\alpha_0\geq1$ will be addressed (together with the faster decay of higher order spherical harmonic) in the next Section~\ref{section.faster}. The main result is the following proposition:
	
	\begin{prop}[Limit at $R=0$] \label{prop.decay}
		Let $\Phi \in H^2_{sol}$ the unique weak solution of \eqref{CSF.4}, arising from initial data $(\snabla_{\mathbb{S}^2}^n\Phi^0,\snabla_{\mathbb{S}^2}^n\dot{\Phi}^0) \in H^2_{data}$, for all $n\leq 2$.
		Then,
		\begin{equation*}
			\lim_{R\downarrow 0}R^{-\alpha_{0}}\Phi(T,R,\omega)
		\end{equation*}
		is well-defined and is equal to $P_{0}(T)$, with $P_{0}\in C^0([-1,0))$, so $R^{-\alpha_{0}}\Phi$ can be extended as a continuous function to $\{R=0\}$. Furthermore, there exists a constant $C(\alpha_0)>0$ such that 
		\begin{align*}
			||P_{0}||_{L^{\infty}([-1,0] )}\leq &\: C   ||(\Phi^0,\dot \Phi^0)||_{H^2_{\rm data}},\\
			||R^{-\alpha_{0}}\Phi-P_0||_{L^{2}(S^2_{T,R})}\leq &\: CR^{1-\alpha_0} ||(\Phi^0, \dot \Phi^0)||_{H^2_{\rm data}},\\
			|R^{-\alpha_{0}}\Phi(T,R,\omega)-P_0(T)|\leq &\: CR^{1-\alpha_0} \sum_{k=0}^2||(\snabla_{\mathbb{S}^2}^k\Phi^0, \snabla_{\mathbb{S}^2}^k\dot \Phi^0)||_{H^2_{\rm data}}.
		\end{align*}		
	\end{prop}

	Before proving Proposition~\ref{prop.decay}, we give a preliminary embedding lemma.
	\begin{lemma}\label{lem.pointwise}
		\begin{enumerate}[i)]
			Let $f\in L^{\infty}([-1,0])_T; \underline{H}^2_{0,0}(\Sigma_T))\cap W^{1,\infty}([-1,0]_T;\underline{H}^1(\Sigma_T))$, such that moreover\\$\snabla_{\mathbb{S}^2}^k \partial_Tf\in~L^{\infty}([-1,0]_T; \underline{H}^1(\Sigma_T))$ for $1 \leq k \leq 2$ and assume that $supp(f(T,\cdot))\subset \{R\leq 2\}$ for all $-1\leq T\leq 0$. Then,

			\item \label{pointwise1} we have that $f\in C^0([-1,0]\times [0,2)\times \s^2)$ and $R^{1-\alpha_0}\partial_R^{\alpha_0}f\in C^0([-1,0]\times [0,2)\times \s^2)$ and there exists $C(\alpha_0)>0$ such that for all $ R\geq 0$, $\omega \in \mathbb{S}^2$
			\begin{align}		\label{eq:pointwPsi}		 |f|(T,R,\omega)\leq &\:C R^{\alpha_0} ||f||_{\underline{H}^2_{0,0}(\Sigma_T)},\\		\label{eq:pointwdPsi}		|\partial_R^{\alpha_0}f|(T,R,\omega)\leq&\: C R^{-1+\alpha_0}\sum_{k=0}^2||\snabla_{\s^2}^kf||_{\underline{H}^2_{0,0}(\Sigma_T)}.		\end{align}

			\item \label{pointwise.0} There exists $C(\alpha_0)>0$ such that for all $R\geq 0$, \begin{equation}\label{averaged.decay}
				\left| \int_{\mathbb{S}^2} f(T,R,\omega) d\omega \right| \leq C \| f\|_{\HH}.
			\end{equation}
			
			\item \label{pointwise2}   Consider the function $F: [-1,0]_T\times \s^2\to \R$ defined as follows:
			\begin{equation}
				\label{F.def} 
				F(T,\omega)=\lim_{R\downarrow 0} R^{1-\alpha_0}\partial_R^{\alpha_0}f(T,R,\omega).
			\end{equation}
			Then $F \in C^0 ([-1,0]_T\times \s^2)$ and there exists a constant $C(\alpha_0)>0$ such that 
			\begin{equation} \label{f.cont.est}
				\left| R^{-\alpha_0} f(R,T,\omega)- \frac{F(T,\omega)}{2\alpha_0}\right| \leq C R^{1-\alpha_0} \sum_{k=0}^2||\snabla_{\s^2}^kf||_{\underline{H}_{0,0}^2(\Sigma_T)}.
			\end{equation}

			\item \label{pointwise3} 	 $R^{-\alpha_0} f$ is a continuous function on $[-1,0]_T\times [0,\infty)_R\times \s^2$.
		\end{enumerate}
	\end{lemma}
	
	\begin{proof} Let $\{f_m\}$ be a sequence in $C_c^{\infty}([-1,0]\times (0,2)\times \s^2)$, such that for all $0\leq k\leq 2$: $\snabla_{\s^2}^kf_m\to \snabla_{\s^2}^kf$, with respect to $ L^{\infty}([-1,0)_T; \underline{H}^2_{0,0}(\Sigma_T))$ and $\snabla_{\s^2}^k\partial_T f_m\to \snabla_{\s^2}^k \rd_T f$, with respect to $ L^{\infty}([-1,0)_T; \underline{H}^1(\Sigma_T))$.

		We apply the fundamental theorem of calculus, integrating from $R=+\infty$, together with Cauchy--Schwarz to estimate
		\begin{equation*}
			\begin{split}
				|R^{1-\alpha_0}\partial_R^{\alpha_0}f_m|^2(T,R,\omega)=&\:\left(\int_R^{+\infty}|\partial_R(R'^{1-\alpha_0}\partial_R^{\alpha_0}f_m)|(T,R',\omega)\,dR'\right)^2\\
				=&\: \left(\int_R^{+\infty} R'^{1-\alpha_0}|\partial_R^{1-\alpha_0}\partial_R^{\alpha_0}f_m|(T,R',\omega)\,dR'\right)^2\\
				\leq &\:\int_R^{2} R'^{1-2\alpha_0}\,dR' \int_R^{+\infty} R'|\partial_R^{1-\alpha_0}\partial_R^{\alpha_0}f_m|^2(R',\omega)\,dR'\\
				\leq & C \int_R^{+\infty} R'|\partial_R^{1-\alpha_0}\partial_R^{\alpha_0}f_m|^2(T,R',\omega)\,dR',
			\end{split}
		\end{equation*}
		where we used that $\alpha_0<1$ to arrive at the last inequality.
		
		Using that $||\partial_R^{1-\alpha_0}\partial_R^{\alpha_0}f_m||_{\underline{L}^2(\Sigma_T)}\leq ||f||_{\underline{H}_{0,0}^2(\Sigma_T)}$, integrating over $\mathbb{S}^2$ and then applying a Sobolev inequality on $\s^2$, we obtain \eqref{eq:pointwdPsi} for $f_m$.			
		
		Now we apply the fundamental theorem of calculus again, integrating from $R=0$ (note that by assumption, we will pick no boundary terms at $\{R=0\}$), to estimate
		\begin{equation*}
			\begin{split}
				|R^{\alpha_0}f_m|(R,\omega)=&\:\int_0^{R}|\partial_R(R'^{\alpha_0}f_m)|(R',\omega)\,dR'\\
				=&\:\int_0^{R}R'^{\alpha_0}|\partial_R^{\alpha_0}f_m|(R',\omega)\,dR'\\
				\leq &\: \int_0^R R'^{-1+2\alpha_0}\,dR' \sup_{R'\in (0,R)} |R^{1-\alpha_0}\partial_R^{\alpha_0}f_m|(R',\omega)\\
				\leq &\: C R^{2\alpha_0} \sum_{k=0}^2||\snabla_{\s^2}^kf_m||_{\underline{H}_{0,0}^2(\Sigma_T)}.
			\end{split}
		\end{equation*}
		This concludes the proof of \eqref{eq:pointwPsi} for $f_m$. By the above uniform estimates for $R^{1-\alpha_0}\partial_R^{\alpha_0}f_m$ and $f_m$ and the assumed convergence property of $f_m$, we conclude that $\{f_m\}$ and $\{R^{1-\alpha_0}\partial_R^{\alpha_0}f_m\}$ are Cauchy sequences in $C^0([-1,0]\times [0,2)\times \s^2)$, from which point \emph{\ref{pointwise1})} follows.
		
		We turn to the proof of \emph{\ref{pointwise.0})}. We can write, using that $f_m=0$ along $\{R=0\}$: 
		\begin{equation*}
			\begin{split}
				R^{\alpha_0}  \left|\int_{\mathbb{S}^2} f_m(R,T,\omega) d\omega\right|=&\:  \left|\int_{\mathbb{S}^2}\int_0^{R} \partial_R (R'^{\alpha_0}f_m)(R',T,\omega) d\omega dR' \right| \leq  \int_{\mathbb{S}^2}  \int_{0}^R R'^{\alpha_0} |\partial_R^{\alpha_0}f_m|(R',T,\omega) d\omega dR'\\
				\leq &\: \frac{ R^{\alpha_0}}{\sqrt{\alpha_0}} \| f_m\|_{\HH},
			\end{split}
		\end{equation*}
		where we applied a Cauchy--Schwarz inequality to arrive at the last inequality. Then \eqref{averaged.decay} follows also in the limit $m\to \infty$, using that, by point \emph{\ref{pointwise1})}, $f\in C^0([-1,0]\times [0,2)\times \s^2)$. We conclude \emph{\ref{pointwise.0})}

		For point \emph{\ref{pointwise2})}, note first that $F$ given by \eqref{F.def} is well-defined and continuous by point \emph{\ref{pointwise1})}. Then, note that	
		\begin{align*}
			|R^{1-\alpha_0} \partial_R^{\alpha_0}f(R,T,\omega)-  F(T,\omega) |=	\left| \int_0^R  \partial_R(R'^{1-\alpha_0} \partial_R^{\alpha_0}f(R',T,\omega)) dR'\right|\leq C R^{1-\alpha_0} \sum_{k=0}^2||\snabla_{\s^2}^kf||_{\underline{H}_{0,0}^2(\Sigma_T)},
		\end{align*} repeating the estimates in the derivation of \eqref{eq:pointwdPsi} (in the interval $[0,R]$ instead of $[R,\infty)$). The above inequality holds also when taking the $L^2(S^2_{T,R})$ on the LHS, in which case we can omit the $\snabla_{\s^2}^k$ on the RHS.
		
		We then apply the above inequality to obtain:
		\begin{equation*}
			\begin{split}
				\left|R^{\alpha_0} f(R,T,\omega)- \frac{R^{2\alpha_0} }{2\alpha_0}  F(T,\omega) \right|=&\:	\left| \int_0^R R'^{2\alpha_0-1}  [ R^{1-\alpha_0} \partial_R^{\alpha_0}f(R',T,\omega)- F(T,\omega) ] dR'\right|\\
				\leq &\:C R^{1+\alpha_0} \sum_{k=0}^2||\snabla_{\s^2}^kf||_{\underline{H}_{0,0}^2(\Sigma_T)},
			\end{split}
		\end{equation*} which gives directly \eqref{f.cont.est} and concludes the proof of point \emph{\ref{pointwise2})}.

		The continuity of $R^{-\alpha} f$ at any point $(R=R_0,T,\omega)$ follows from \emph{\ref{pointwise1})}. Continuity at $(R=0,T,\omega)$ follows from continuity of $F$ together with \eqref{f.cont.est}.
	\end{proof}
	\begin{proof}[Proof of Proposition~\ref{prop.decay}]
		Now we return to the proof of Proposition~\ref{prop.decay}. By Proposition~\ref{prop.H2}, setting $f = \Phi$, we know for all $T \in [-1,0]$ the assumptions of Lemma~\ref{lem.pointwise} are satisfied. Furthermore, since $\alpha_0<1$ and $\Phi_{\geq 1}(T,\cdot) \in \underline{H}^2(\Sigma_T)$ for almost all $T\in [-1,0]$, we must have $\pi_{\geq 1}(F)\equiv 0$. The conclusion of Proposition~\ref{prop.decay} follows.
		
	\end{proof}

	\subsection{Refined estimates for higher harmonics or $\alpha_0\geq1$}\label{section.faster}
	In this section, we show that restricting to higher spherical harmonics $\Phi_{\geq \ell}$ allows us to derive improved elliptic estimates, which in turn result in faster decay in $R$ as $R\downarrow 0$. The estimates in this section are also necessary to obtain quantitative estimates near $R=0$ for the full solution $\Phi$ of \eqref{wave.pot3} with $\alpha_0\geq1$. We will start with a technical embedding lemma.

	\begin{lemma} 
		
		Let $p \in \R$ such that $p+1 < \alpha_{\ell+1}$. Assume that $R^{-p} f \in \pi_{\geq \ell}(\LL)$ and let $\Phi_{\geq \ell} \in~\HH$ be the corresponding weak solution of \eqref{def.Dirichlet}. Assume that $supp(\Phi_{\geq \ell}) \subset \{R< 2\}$. Then there exists $C(\ell)>0$ such that for all $R_0\leq 2$

			\begin{align}
				\label{ptwise3}  \| R^{-(p+1)} \Phi_{\geq\ell+1} \|_{L^{2}(S_{T,R_0})}\leq &\:C \| \Phi_{\geq \ell+1} \|_{\underline{H}^2_{\ell,p}(\Sigma_T)} \leq C\| R^{-p}f_{\geq \ell+1}\|_{\LL},\\
				\label{ptwise3b} 
				\| R^{-\min\{p+1,\alpha_{\ell}\}} \Phi_{\ell} \|_{L^{\infty}(\Sigma_T)}\leq  &\:C \| \Phi_{ \ell} \|_{\underline{H}^2_{\ell,p}(\Sigma_T)} \leq C \| R^{-p} f_{\ell}\|_{\LL}
			\end{align}

	\end{lemma}
	
	\begin{proof}	To prove
	\eqref{ptwise3} and \eqref{ptwise3b}, we proceed analogously to the proof of Lemma \ref{lem.pointwise}. Without loss of generality, we may derive estimates for $\Phi_{\geq \ell}=w_{\geq \ell} \in~ \pi_{\geq \ell} ( C^{\infty}_c(\Sigma_T\cap\{R<2\}))$ and conclude \eqref{ptwise3} and \eqref{ptwise3b} via a density argument.
	
	We first consider $\Phi_{\geq \ell+1}$ and note that by the fundamental theorem of calculus in the $R$-direction together with Young's inequality, we obtain
	\begin{equation*}
		\begin{split}
			\int_{\s^2}R^{-2(p+1)}|\Phi_{\geq \ell+1}|^2(T,R,\omega)\,d\sigma_{\s^2}\leq &-\int^2_{R}\int_{\s^2}\partial_R(R^{-2(p+1)}|\Phi_{\geq \ell+1}|^2)(T,R',\omega)\,d\sigma_{\s^2}dR'\\
			\leq &\: C \int_{\Sigma_T} \bigl[R'^{-2(p+1)-1}|\Phi_{\geq \ell+1}|^2+R'^{-2(p+1)+1}|\partial_R\Phi_{\geq \ell+1}|^2\bigr]d\sigma_{\s^2}dR'\\
			\leq &\: C ||R^{-(p+1)}\Phi_{\geq \ell+1}||^2_{\underline{H}^1(\Sigma_T)}\leq  C ||R^{-p}\Phi_{\geq \ell+1}||^2_{\underline{H}^2(\Sigma_T)}.
		\end{split}
	\end{equation*}
	Then \eqref{ptwise3} follows from \eqref{est.elliptic.ang2} with $\ell$ replaced by $\ell+1$.
	
	We now consider $\Phi_{\ell}=w_{\ell}$ and apply the fundamental theorem of calculus again, together with Cauchy--Schwarz, to obtain:
	\begin{equation*}
		\begin{split}
			|R^{1-\alpha_{\ell}}\partial_R^{\alpha_{\ell}}w_{ \ell}|^2(R,T,\omega)\leq&\: \int_R^{2} R'^{1-2\alpha_{\ell}+2p}\,dR' \int_R^{\infty}  R'^{1-2p}|\partial_R^{1-\alpha_{\ell}}\partial_R^{\alpha_{\ell}}w_{\ell}|^2(R',T,\omega)\,dR'\\
			\leq  &\: CR^{2\min\{p+1-\alpha_{\ell},0\}} \int_R^{+\infty}  R'^{1-2p}|\partial_R^{1-\alpha_{\ell}}\partial_R^{\alpha_{\ell}}w_{\ell}|^2(R',T,\omega)\,dR'.
		\end{split}
	\end{equation*}
	We then apply the Sobolev inequality on $\s^2$ (using the fact that $w_{\ell}$ is supported on a finite number of spherical harmonics) to obtain for some $C=C(\ell)>0$:
	\begin{equation}
		\label{eq:pointwdwl}
		\begin{split}
			\sup_{\omega \in \s^2}|R^{1+\max\{-\alpha_{\ell},-(p+1)\}}\partial_R^{\alpha_{\ell}}w_{ \ell}|^2(R,T,\omega)\leq    C  \|w_{\ell}\|_{\underline{H}^2_{\ell,p}}^2 .
		\end{split}
	\end{equation}
	Applying the fundamental theorem in $R$ once more, together with \eqref{eq:pointwdwl}, we moreover obtain:
	\begin{equation}
		\label{eq:pointwwl}
		\begin{split}
			\sup_{\omega \in \s^2}|R^{\alpha_{\ell}}w_{\ell}|^2(T,R,\omega)\leq &\:\left(\int_0^R \partial_R(R^{\alpha_{\ell}}w_{\geq \ell})(T,R',\omega)\,dR'\right)^2\\
			\leq &\: \sup_{\omega \in \s^2,R'\leq R}|R'^{1+\max\{-\alpha_{\ell},-(p+1)}\}\partial_R^{\alpha_{\ell}}w_{\geq \ell}|^2(R',T,\omega) \left(\int_0^R R'^{-1+\min\{2\alpha_{\ell},\alpha_{\ell}+(p+1)\}}\,dR'\right)^2\\
			\leq&\:  C \bigl[R^{\min\{2\alpha_{\ell},\alpha_{\ell}+(p+1)\}} \bigr]^2 \| w_{\ell}\|_{\underline{H}^2_{\ell,p}}^2.
		\end{split}
	\end{equation}
	The inequalities \eqref{ptwise3b} then follow from \eqref{est.elliptic.ang}.\qedhere

	\end{proof}
	
	\begin{prop}
		\label{prop:R0limit}
		
		Assume either that $\ell\geq 1$, or that $\ell=0$ and $\alpha_0\geq1$. Suppose $(\Phi^0,\dot \Phi^0)\in \pi_{\geq \ell}(H^2_{data,\lfloor\alpha_{\ell}\rfloor})$. Then there exists a constant $C=C(\ell,\alpha_0)>0$ such that for all $R_0\leq 2$:
		\begin{align}
			\label{eq:keylargeellest}
			|| \Phi_{\geq \ell}||_{H^2_{{\ell},\alpha_{\ell}-1}(\Sigma_T) }\leq &\:C 	 ||( \Phi^0,\dot \Phi^0)_{\geq \ell}||_{H^2_{\rm data,  \lfloor\alpha_{\ell}\rfloor}},\\
			\label{eq:keylargeellest2}
			||R^{-\alpha_{\ell}}\Phi_{\geq \ell}||_{L^{2}(S^2_{T,R})}\leq &\:  C ||(\Phi^0, \dot \Phi^0)_{ \geq \ell}||_{H^2_{\rm data,  \lfloor\alpha_{\ell}\rfloor}}.
		\end{align}
		Furthermore, for any $0<\eta< 1- (\alpha_{\ell}-\lfloor\alpha_{\ell}\rfloor)$, there exists a constant $C=C(\ell,qe,\eta)>0$ such that for all $R_0\leq 2$
		\begin{equation}
			\label{eq:genasympR}
			||R^{-\alpha_{\ell}}\Phi_{\geq \ell}-P_{\ell}||_{L^2(S^2_{T,R})}\leq  CR^{\eta}||(\Phi^0, \dot \Phi^0)_{ \geq\ell}||_{H^2_{\rm data , \lfloor\alpha_{\ell}\rfloor}},
		\end{equation}
		with $P_{\ell}(T,\omega)=\sum_{m=-\ell}^{\ell}P_{\ell m}(T)Y_{\ell m}(\omega)$, for some $P_{\ell m}\in C^0([-1,0])$ that satisfy
		\begin{align}
			\label{eq:boundP}
			||P_{\ell m}||_{L^{\infty}([-1,0])}\leq &\:  C ||(\Phi^0, \dot \Phi^0)_{ \ell}||_{H^2_{\rm data , \lfloor\alpha_{\ell}\rfloor}}.\\
		\end{align}
		If moreover $(\snabla_{\s^2}^k\Phi^0,\snabla_{\s}^k\dot \Phi^0)\in \pi_{\geq \ell}(H^2_{\rm data, \lfloor\alpha_{\ell}\rfloor})$ for $1\leq k\leq 2$, then for all $\omega \in \s^2$ and $R\leq 2$:
		\begin{equation}
			\label{eq:genasympRpoint}
			|R^{-\alpha_{\ell}}\Phi_{\geq \ell}(T,R,\omega)-P_{\ell}(T,\omega)|\leq  CR^{\eta}\sum_{k=0}^2||(\snabla_{\s^2}^k\Phi_0,\snabla_{\s}^k\dot \Phi_0)_{ \geq\ell}||_{H^2_{\rm data ,\lfloor\alpha_{\ell}\rfloor}}.
		\end{equation}
	\end{prop}
	\begin{proof}
		By \eqref{est.elliptic.ang} with $p+1=\alpha_{\ell}<\alpha_{\ell+1}$, together with the expression of $\mathcal{L}\Phi_{\geq \ell}$ in terms of $T$-derivatives of $\Phi_{\geq \ell}$ given by \eqref{CSF.4.5}, we obtain
		\begin{equation}
			\label{eq:auxellestimate}
			\begin{split}
				\|\Phi_{\geq \ell}\|_{H^2_{{\ell},\alpha_{\ell}-1}(\Sigma_T) }\leq&\: C\|R^{-(\alpha_{\ell}-1)}\mathcal{L}\Phi_{\geq \ell}\|_{\underline{L}^2(\Sigma_T) }\leq C \left(||R^{-\alpha_{\ell}}\partial_T\Phi_{\geq \ell}||_{\underline{L}^2(\Sigma_T)}+||R^{-\alpha_{\ell}+1}\partial_T^2\Phi_{\geq \ell}||_{\underline{L}^2(\Sigma_T)}\right).
			\end{split}
		\end{equation}
		If $1< \alpha_{\ell}\leq 2$,  using that $supp(\Phi) \subset \{R\leq 2\}$, the RHS is certainly bounded by
		\begin{equation*}
			||\partial_T\Phi_{\geq \ell}||_{\underline{H}^2(\Sigma_T)}+||\partial_T^2\Phi_{\geq \ell}||_{\underline{H}^1(\Sigma_T)},
		\end{equation*}
		and $||\partial_T\Phi_{\geq \ell}||_{\underline{H}^2(\Sigma_T)}$ is bounded by $||\partial_T\Phi_{\geq \ell}||_{\underline{H}^2_{0,0}(\Sigma_T)}$ after applying Proposition \ref{lemma.elliptic.ang} (ii),	so \eqref{eq:keylargeellest} holds as a consequence of \eqref{H2.est.noang}.
		
		If $\alpha_{\ell}=1$, the RHS of \eqref{eq:auxellestimate} is bounded by
		\begin{equation*}
			||\partial_T\Phi_{\geq \ell}||_{\underline{H}^1(\Sigma_T)}+||\Phi_{\geq \ell}||_{\underline{H}^2_{0,0}(\Sigma_T)}
		\end{equation*}
		and we can apply \eqref{H2.est.noang} again.
		
		Suppose now that $2< \alpha_{\ell}\leq3$. The key point here is that all the power in $R^{-\alpha_{\ell}}$ appearing in front of the $\partial_T\Phi_{\geq \ell}$ term on the RHS of \eqref{eq:auxellestimate} satisfies $(\alpha_{\ell}-2)+1<\alpha_{\ell}$ so we can apply \eqref{est.elliptic.ang2} to estimate the RHS \eqref{eq:auxellestimate} further and obtain:
		\begin{multline*}
			\|\Phi_{\geq \ell}\|_{H^2_{{\ell},\alpha_{\ell}-1}(\Sigma_T) }\leq C \left(||R^{-\alpha_{\ell}+2}\partial_T^3\Phi_{\geq \ell}||_{\underline{L}^2(\Sigma_T)}+||R^{-\alpha_{\ell}+1}\partial_T^2\Phi_{\geq \ell}||_{\underline{L}^2(\Sigma_T)}\right)\\
			\leq C(||\partial_T^3\Phi_{\geq \ell}||_{\underline{H}^1(\Sigma_T)}+||\partial_T^2\Phi_{\geq \ell}||_{\underline{H}^2(\Sigma_T)})\leq C (||\partial_T^3\Phi_{\geq \ell}||_{\underline{H}^1(\Sigma_T)}+||\partial_T^2\Phi_{\geq \ell}||_{\underline{H}^2_{0,0}(\Sigma_T)}),
		\end{multline*} where we used Proposition~\ref{lemma.elliptic.ang} (ii) again. We conclude that \eqref{eq:keylargeellest} holds, again using \eqref{H2.est.noang}. More generally, we need to apply \eqref{est.elliptic.ang2} a total $\lfloor \alpha_{\ell} \rfloor-1$ extra times to obtain \eqref{eq:keylargeellest} for general $\alpha_{\ell}> 2$. 
		
		Subsequently, we obtain \eqref{eq:keylargeellest2} after applying  \eqref{ptwise3b}.
		
		The existence and continuity in $T$ of the limit of the function $R^{-\alpha_{\ell}}\Phi_{\geq \ell}$ at $R=0$, together with the estimate \eqref{eq:genasympR}, follow by repeating the arguments in the proof of (iii) and (iv) of Lemma \ref{lem.pointwise} and making use of the estimates leading to \eqref{eq:keylargeellest2}, but with $p+1=\alpha_{\ell}+\eta$ in \eqref{est.elliptic.ang}, with $\eta>0$ suitably small so that $\lfloor 
		\alpha_{\ell}+\eta\rfloor= \lfloor 
		\alpha_{\ell}\rfloor$. Equivalently, we need that $0<\eta<\lfloor \alpha_{\ell}\rfloor-(\alpha_{\ell}-1)
		$. We omit the details.
		The support on the $\ell$-th mode of the limiting function then follows from the boundedness of $||R^{-\alpha_{\ell}-\eta}\Phi_{\geq \ell+1}||_{L^2(\Sigma_T\cap\{R=0\})}$ at $R=0$, which implies vanishing of $R^{-\alpha_{\ell}}\Phi_{\geq \ell+1}$ at $R=0$.
	
	\end{proof}

	\subsection{Genericity of late-time tails}
	In this section, we show that the limit function $P_{\ell}$ defined in Proposition \ref{prop:R0limit} is non-zero for generic initial data with respect to the initial data topology induced by $||\cdot||_{H^2_{\rm data ,\lfloor\alpha_{\ell}\rfloor}}$.
	\begin{prop} \label{prop.tail} The subspace
		\begin{equation*}
			\mathcal{S}_{\geq \ell}=\left\{(\Phi^0,\dot \Phi^0)_{\geq \ell}\in H^2_{\rm data , \lfloor\alpha_{\ell}\rfloor}\,\Big|\, \lim_{(T,R_0)\to(0,0) } ||\Phi_{\geq \ell}||_{L^2(S_{T,R_0}})=0\right\},
		\end{equation*}
		with $ \Phi_{\geq \ell}$ denoting the unique corresponding weak solution to \eqref{CSF.4}, has codimension at least $2\ell+1$. The complement $\mathcal{S}_{\geq \ell}^c$ is open and dense with respect to $||\cdot||_{H^2_{\rm data , \lfloor\alpha_{\ell}\rfloor}}$.

	\end{prop}
	\begin{proof}
		We first note that the functions $e_{\ell m}(T,R,\omega):=R^{\alpha_{\ell}}Y_{\ell m}(\omega) \chi(R)$ (where $\chi$ is a suitable smooth cut-off with $\chi(R)=1$ for $R\leq 2$)  for $m\in \mathbb{Z}$ such that $|m|\leq \ell$ satisfy \eqref{CSF.4.5} for $R\leq 2$ and arise from the initial data pair $(e_{\ell m},0)\in H^2_{\rm data , \lfloor\alpha_{\ell}\rfloor}$.
		
		For any initial data pair $(\Phi_0,\dot \Phi_0)_{\geq \ell} \in H^2_{\rm data ,\lfloor\alpha_{\ell}\rfloor}$ we can consider the $2\ell+1$-parameter family of initial data of the form: let $\lambda=(\lambda_{-\ell},\ldots, \lambda_{\ell})\in \mathbb{C}^{2\ell+1}$ and define
		\begin{equation*}
			(\Phi_0^{\lambda},\dot \Phi_0^{\lambda})=(\Phi_0,\dot \Phi_0)+\left(\sum_{m=-\ell}^{\ell} \lambda_m e_{\ell m},0\right).
		\end{equation*}
		By linearity of \eqref{CSF.4.5} and Proposition \ref{prop.existence.wave}, there exists a corresponding unique weak solution $\Phi^{\lambda}_{\geq \ell}\in \pi_{\geq \ell}H^1_{sol}$ with
		\begin{equation*}
			\Phi^{\lambda}_{\geq \ell}=\Phi_{\geq \ell}+\sum_{m=-\ell}^{\ell} \lambda_m e_{\ell m},
		\end{equation*}
		where $\Phi$ is the unique weak solution corresponding to data $(\Phi_0,\dot \Phi_0)$. Now we apply \eqref{eq:genasympR} to conclude that for all $R_0\leq 2$:
		\begin{equation*}
			\left|\left|R^{-\alpha_{\ell}}\Phi_{\geq \ell}^{\lambda}-\sum_{m=-\ell}^{\ell} (P_{\ell m}(T)-\lambda_{m}) Y_{\ell m}\right|\right|_{L^2(S_{T,R_0})}\leq  CR^{\eta}||(\Phi_0,\dot \Phi_0)_{ \geq\ell}||_{H^2_{\rm data , \lfloor \alpha_{\ell}\rfloor}}.
		\end{equation*}
		Hence, we have that
		\begin{equation*}
			\lim_{(T,R)\to (0,0)}||R^{-\alpha_{\ell}}\Phi_{\geq \ell}^{\lambda}||_{L^2(\Sigma_T\cap\{R=R_0\})}=0
		\end{equation*}
		if and only if $\lambda_{\ell m}=P_{\ell m}(0)$ for all $m\in\{-\ell,-\ell+1,\ldots,\ell-1,\ell\}$. Therefore, $\mathcal{S}_{\geq \ell}$ is a subspace with codimension \emph{at least} $2\ell+1$.
		
		Openness of $\mathcal{S}_{\geq \ell}^c$ follows from \eqref{eq:boundP} and linearity of \eqref{CSF.4}. Density follows from the codimension property of $\mathcal{S}_{\geq \ell}$.

	\end{proof}

	\subsection{Proof of Theorems \ref{main.theoremell0} and \ref{main.theorem}}
	In this section, we prove Theorem \ref{main.theoremell0} and Theorem \ref{main.theorem}. 
	
	Consider first the case $\alpha_0<1$ with initial data $(\phi^0,\dot \phi^0)$ $\in \mathcal{H}^2_{data,1}$.

	Let $(u,v,\omega)\in D^+(\Sigma_{-1})$. Then there exists a $\delta>0$ suitably small, such that the solution $\phi$ corresponding to the above initial data satisfies
	\begin{equation*}
		r\phi(u,v,\omega)=R^{\frac{1}{2}+i qe }(u,v)\Phi(u,v,\omega),
	\end{equation*}
	with $\Phi$ the solution to \eqref{CSF.4} arising from the initial data $(\Phi^0,\dot \Phi^0)$ that are related to $(\phi^0, \dot \phi^0)$ as described in Lemma \ref{lemma.conversion}. In the setting of \eqref{ISP}, we can simply set $qe=0$.

	We can then estimate via Proposition \ref{prop.decay}:
	\begin{equation}
		\label{eq:mainconvest1}
		\begin{split}
			||r\phi(u,v,\cdot)-R^{\frac{1}{2}+ieq+\alpha_{0}}(u,v)P_0(0)||_{L^{2}(S^2_{u,v})}\leq&\:  C|R|^{\frac{3}{2}}(u,v)||(\Phi^0,\dot \Phi^0)||_{\underline{H}^2_{\rm data}}\\
			&+C|R|^{\alpha_0+\frac{1}{2}}|T|(u,v)||(\partial_T \Phi|_{T=-1},\partial_T^2 \Phi|_{T=-1})||_{\underline{H}^2_{\rm data}},
		\end{split}
	\end{equation}
	where we additionally used that by the fundamental theorem of calculus:
	\begin{equation*}
		|P_0(T)-P_0(0)|\leq |T| ||\partial_TP_{0}||_{L^{\infty}([-1,0] )}\leq C|T|   ||(\partial_T \Phi|_{T=-1}, \partial_T^2 \Phi|_{T=-1})||_{\underline{H}^2_{\rm data}}.
	\end{equation*}
	Recall that
	\begin{align}
		\label{eq:Ruv}
		R(u,v)=&\:u^{-1}-v^{-1}=\frac{v-u}{uv}=\frac{4r}{(t-r)(t+r)},\\
		\label{eq:Tuv}
		|T|(u,v)=&\:u^{-1}+v^{-1}=\frac{v+u}{uv}=\frac{4(t-r)+4r}{(t-r)(t+r)}.
	\end{align}

	We conclude (i) of Theorem \ref{main.theoremell0} for $\alpha_0<1$ with $Q_0=4^{\frac{1}{2}+ieq+\alpha_0}P_0(0)$ after estimating the RHS of \eqref{eq:mainconvest1} as follows via Lemma \ref{lemma.conversion}:
	\begin{equation*}
		\begin{split}
			\sum_{0\leq k\leq 1} ||(\partial_T^k\Phi|_{T=-1},\partial_T^{k+1}\Phi|_{T=-1})||_{\underline{H}^2_{\rm data}}\leq C \|(\phi^0, \dot \phi^0)\|_{\mathcal{H}^2_{\rm data,1}	}.
		\end{split}
	\end{equation*}
	To obtain (ii) for $\alpha_0<1$, we assume moreover that $(\snabla_{\s^2}^lK^{k}(r\pi_{\geq 1}\phi)|_{\Sigma_{-1}},\snabla_{\s^2}^lK^{k+1}(r\pi_{\geq 1}\phi)|_{\Sigma_{-1}})\in \mathcal{H}^2(\Sigma_{-1})\times  \mathcal{H}^1(\Sigma_{-1})$ for all $0\leq k \leq 1$ and $0\leq l\leq 2$ and apply the Sobolev embedding on $\s^2$.
	
	We now turn to (i) of Theorem \ref{main.theorem}, from which (i) of Theorem \ref{main.theoremell0} with $\alpha_0\geq1$ immediately follows. We assume that
	\begin{equation*}
		(\phi^0,\dot \phi^0)\in  \mathcal{H}^{2}_{\rm data,\lfloor \alpha_{\ell} \rfloor+1}.
	\end{equation*}
	Then we apply Proposition \ref{prop:R0limit} to obtain:
	\begin{equation}
		\label{eq:mainconvest2}
		\begin{split}
			||r\phi_{\geq \ell}(u,v,\cdot)-R^{\frac{1}{2}+ieq+\alpha_{\ell}}(u,v)P_{\ell}(0,\cdot)||_{L^{2}(S^2_{u,v})}\leq &\: CR^{\frac{1}{2}+\lfloor \alpha_{\ell}\rfloor+1-\epsilon}(u,v)||\pi_{\geq \ell}(\Phi^0,\dot \Phi^0)||_{\underline{H}^2_{\rm data,\lfloor\alpha_{\ell}\rfloor}}\\
			&+CR^{\frac{1}{2}+\alpha_{\ell}}|T|(u,v)\|\pi_{\geq \ell}(\dot \Phi^0,\partial_T^2 \Phi|_{T=-1})\|_{\underline{H}^2_{\rm data,\lfloor\alpha_{\ell}\rfloor}},
		\end{split}
	\end{equation}
	with $\epsilon>0$ arbitrarily small, where we additionally used that by the fundamental theorem of calculus:
	\begin{equation*}
		|P_{\ell}(T,\omega)-P_{\ell}(0,\omega)|\leq |T| ||\partial_TP_{\ell}||_{L^{\infty}([-1,0]\times \s^2 )}\leq C |T|  ||(\partial_T \Phi|_{T=-1}, \partial_T^2 \Phi|_{T=-1})||_{\underline{H}^2_{\rm data,\lfloor\alpha_{\ell}\rfloor}}.
	\end{equation*}
	We conclude (i) of Theorem \ref{main.theorem} with $Q_{\ell}=4^{\frac{1}{2}+ieq+\alpha_{\ell}}P_{\ell}(0,\cdot)$ after applying \eqref{eq:Ruv} and \eqref{eq:Tuv}, and then estimating the RHS of \eqref{eq:mainconvest2} as follows via Lemma \ref{lemma.conversion}:
	\begin{equation*}
		\begin{split}
			\sum_{0\leq k\leq 1}&  ||\pi_{\geq \ell}(\partial_T^k \Phi|_{T=-1}, \partial_T^{k+1} \Phi|_{T=-1})||_{\underline{H}^2_{\rm data,\lfloor\alpha_{\ell}\rfloor} }\\
			\leq &\:C||(\phi^0,\dot\phi^0)||_{\mathcal{H}^{2}_{\rm data,\lfloor \alpha_{\ell} \rfloor+1}}.
		\end{split}
	\end{equation*}
	To prove (ii) of Theorem \ref{main.theorem}, we assume moreover that $(\phi^0,\dot \phi^0)\in \mathcal{H}^{2}_{\rm data,\lfloor \alpha_{\ell} \rfloor+1,\snabla}$ for all $0\leq k \leq 1$ and $0\leq l\leq 2$ and apply Sobolev on $\s^2$. This also immediately implies (ii) of Theorem \ref{main.theoremell0} in the case $\alpha_0>1$.
	
	Finally, we can conclude (iii) of both Theorem \ref{main.theoremell0} and Theorem \ref{main.theorem}, by noting that $Q_{\ell}=4^{\frac{1}{2}+ieq+\alpha_{\ell}}P_{\ell}(0,\cdot)$ for all $\ell\geq 0$ and applying Proposition \ref{prop.tail}.
	
	\bibliographystyle{plain}
	\bibliography{bibliography.bib}

\end{document}